\documentclass[10pt]{article}
\usepackage[utf8]{inputenc}

\usepackage{url}
\usepackage{amsfonts}
\usepackage{amsmath}
\usepackage{amssymb}
\usepackage{amsthm}
\usepackage{ifthen}
\usepackage{graphicx}

\newtheorem{theorem}{Theorem}[section]
\newtheorem{proposition}[theorem]{Proposition}
\newtheorem{corollary}[theorem]{Corollary}
\newtheorem{lemma}[theorem]{Lemma}
\newtheorem{problem}[theorem]{Problem}
\theoremstyle{remark}

\newtheorem{example}[theorem]{Example}
\theoremstyle{definition}

\makeatletter
\newcommand{\nolisttopbreak}{\vspace{\topsep}\nobreak\@afterheading}
\makeatother
\newenvironment{listproof}[1][\proofname]{\begin{proof}[#1]\mbox{}\nolisttopbreak}{\end{proof}}

\newcommand{\parder}[3][Default]{
	\frac{\partial \ifthenelse{\equal{#1}{Default}}{}{^{#1}}#2}{
              \partial #3 \ifthenelse{\equal{#1}{Default}}{}{^{#1}}}}
\newcommand{\bcdot}{$\discretionary{\mbox{$ \cdot $}}{}{}$}
\newcommand{\sdots}{\vdots\,\vdots\,\vdots\,}

\newcommand{\jac}{{\mathcal J}}

\newcommand{\GL}{\operatorname{GL}}

\newcommand{\trdeg}{\operatorname{trdeg}}
\newcommand{\traildeg}{\operatorname{codeg}}

\newcommand{\tr}{\operatorname{tr}}
\newcommand{\imp}{{\mathversion{bold}$\Rightarrow$} }
\newcommand{\C}{{\mathbb C}}

\newcommand{\R}{{\mathbb R}}
\newcommand{\Q}{{\mathbb Q}}

\newcommand{\Z}{{\mathbb Z}}
\newcommand{\tp}{^{\rm t}}

\newcommand{\rk}{\operatorname{rk}}
\newcommand{\Mat}{\operatorname{Mat}}

\allowdisplaybreaks
\numberwithin{equation}{section}

\title{Rational maps $H$ for which $K(tH)$ has transcendence degree 2 over $K$}
\author{Michiel de Bondt\footnote{The author was supported by the Netherlands 
        Organisation for Scientific Research (NWO).} \\
        Institute for Mathematics, Astrophysics and Particle Physics \\
        Radboud University Nijmegen \\
        \emph{Email address:} M.deBondt@math.ru.nl}

\begin{document}

\maketitle

\begin{abstract}
We classify all rational maps $H \in K(x)^n$ for which $\trdeg_K K(tH_1, tH_2, \allowbreak 
\ldots, tH_n) \le 2$, where $K$ is any field and $t$ is another indeterminate. 

Furthermore, we classify all such maps for which additionally 
$\jac H \cdot H = \tr \jac H \cdot H$ (where $\jac H$ is the Jacobian matrix of $H$), i.e.\@
$$
\sum_{i=1}^n H_i \parder{}{x_i} H_k = \sum_{i=1}^n H_k \parder{}{x_i} H_i
$$
for all $k \le n$. This generalizes a theorem of Paul Gordan and Max N{\"o}ther, in which
both sides and the characteristic of $K$ are assumed to be zero.

Besides this, we use some of our tools to obtain several results about $K$-subalgebras 
$R$ of $K(x)$ for which $\trdeg_K L = 1$, where $L$ is the fraction field of $R$.

We start with some observations about to what extent, L{\"u}roth's theorem can be generalized.
\end{abstract}

quasi-translation, Jacobian matrix.


\section{Introduction}

Throughout this paper, $x = (x_1,x_2,\ldots,x_n)$ denotes an $n$-tuple of variables
and $y = (y_1,y_2,\ldots,y_m)$ denotes an $m$-tuple of variables, where $n,m \ge 1$.
Furthermore, we use $K$ to denote an arbitrary field and $K[x]$ and $K[y]$ for the polynomial
rings over $K$ which are generated by the components of $x$ and $y$ respectively, i.e.\@
$K[x] = K[x_1,x_2,\ldots,x_n]$ and $K[y] = K[y_1,y_2,\ldots,y_m]$. Denote by $K(x)$ the 
fraction field of $K[x]$. Then $K[x] \subseteq K(x)$.

We denote the $m$-tuples of elements of a set $S$ by $S^m$. If $G \in K(x)^m$,
then $G$ is of the form $G = (G_1,G_2,\ldots,G_m)$ such that $G_i \in K(x)$ for all $i$,
and we can take the \emph{Jacobian matrix} $\jac G$ of $G$ (with respect to $x$), 
which is defined by
$$
\jac G = \jac_x G = \left( \begin{array}{cccc}
\parder{}{x_1} G_1 & \parder{}{x_2} G_1 & \cdots & \parder{}{x_n} G_1 \\
\parder{}{x_1} G_2 & \parder{}{x_2} G_2 & \cdots & \parder{}{x_n} G_2 \\
\vdots & \vdots & \sdots & \vdots \\
\parder{}{x_1} G_m & \parder{}{x_2} G_m & \cdots & \parder{}{x_n} G_m
\end{array} \right)\mbox{.}
$$
If $n=m$, then we call the determinant of this matrix the \emph{Jacobian determinant} of $G$ 
(with respect to $x$).

If $g \in K(x)$, then the Jacobian matrix $\jac g$ of $g$ (with respect to $x$)
is just the Jacobian matrix of $g$ as a one-tuple. If 
$b = (b_1, b_2, \ldots, b_m)$ is an $m$-tuple, then we write $ab$ as an abbreviation of 
$(ab_1, ab_2, \ldots, ab_m)$, where $a$ may be anything, such as a variable $t$, 
a polynomial $g$, or $\parder{}{y_1}$. 

We denote the degree of the term with the largest (smallest) degree of $G$ by
$\deg G$ ($\traildeg G$). Here, $G$ may be a polynomial map or just a single polynomial.
To ensure that $\deg G$ ($\traildeg G$) is always defined, we set 
$$
\deg 0 := -\infty \qquad \mbox{and} \qquad \traildeg 0 := +\infty\mbox{.}
$$
We write $a|_{b=c}$ for the substitution of $b$ by $c$ in $a$. Here, $b$ represents a variable 
or a sequence of variables. Furthermore, $c$ may depend on $b$, so it matters that the 
substitution is performed only one time. 

\medskip
Suppose that $R$ is an integral domain, with fraction field $L$.
We say that $H \in L^m$ is {\em primitive} (over $R$), if $H \ne 0$ and
\begin{equation} \label{primdef}
g^{-1} H \in R^m ~\Longleftrightarrow~ g^{-1} \in R
\end{equation}
for every nonzero $g \in R$. We say that $H \in L^m$ is {\em superprimitive} (over $R$),
if $H \ne 0$ and \eqref{primdef} even holds for every nonzero $g \in L$. 
The reader may prove the following:
\begin{enumerate}

\item[(i)] $H$ is primitive, if and only if $H \in R^m$ and every common divisor
of $H_1, H_2, \allowbreak \ldots, H_m$ is a divisor of $1$, i.e.\@ is a unit in $R$.

\item[(ii)] $H$ is superprimitive, if and only if $H \in R^m$ and for every $r \in R$,
every common divisor of $rH_1, rH_2, \ldots, rH_m$ is a divisor of $r$. 

\end{enumerate}
Notice that the implication $\Longleftarrow$ in \eqref{primdef} is trivially satisfied
if $H \in R^m$, so we can replace $\Longleftrightarrow$ by $\Longrightarrow$
in \eqref{primdef} if we restrict the definition of (super)primitive to elements of $R^m$.

We use $t$ to denote a single variable 
and $R[t]$ and $R(t)$ to denote the corresponding polynomial ring over $R$ and its 
fraction field respectively. We say that a polynomial in $L[t]$ is 
{\em (super)primitive} if the sequence of its coefficients is (super)primitive.

We say that $R$ has the \emph{PSP-property} if every primitive element of $R^m$
is superprimitive. Notice that we may replace $R^m$ by $L^m$, $R[t]$ or $L[t]$ in
the definition of the PSP-property. A \emph{PSP-domain} is an integral domain which 
satisfies the PSP-property. 

If $H \in L^m$, then we call $\tilde{H}$ a (super)primitive part of $H$ if $\tilde{H}$
is (super)primitive and dependent over $L$ of $H$. Now suppose that $H \in L^m$ and
that $\tilde{H}$ is a superprimitive part of $H$, and say that $\tilde{H} = g^{-1}H$, where
$g \in L$. If $H \in R^m$, then we can derive from $g \tilde{H} = H \in R^m$ and the 
fact that $\tilde{H}$ is superprimitive that $g \in R$. 

If $H$ is primitive, then we can deduce from $g^{-1} H = \tilde{H} \in R^m$ that 
$g^{-1} \in R$, so $g$ is a unit in $R$. In general, any primitive part of $H$ is equal 
to $\tilde{H}$ up to multiplication by a unit in $R$. So it is justified to talk about 
\emph{the} primitive part of $H$ if $H$ has a superprimitive part. 

If $H \in R^m$, then we call $g$ a \emph{greatest common divisor} 
if any common divisor of the components of $H$ is also a common divisor of $g$. 
One can check that a greatest common divisor is unique up to multiplication 
by a unit, and that the following statements are equivalent:
\begin{enumerate}

\item[(1)] $H = 0$ or $H$ has a superprimitive part.

\item[(2)] Every element of $R^m$ which is dependent over $L$ of $H$ has a greatest
common divisor.
 
\end{enumerate}
We call $R$ a \emph{GCD-domain} if every element of $R^2$ has a greatest common divisor.
One can verify that for elements $a_i$ in an integral domain, the following holds:
if $\gcd\{a_m,a_{m+1}\}$ exists, then
$$
\gcd\{a_1,a_2,\ldots,a_{m-1},\gcd\{a_m,a_{m+1}\}\} = 
\gcd\{a_1,a_2,\ldots,a_{m-1},a_m,a_{m+1}\}\mbox{,}
$$
provided either the left-hand side or the right-hand side exists. Using that, 
one can show that for a GCD-domain $R$ and any positive integer $m$, 
every element of $R^m$ has a greatest common divisor, or equivalently, every nonzero 
$H \in L^m$ has a superprimitive part. 

On account of the essential uniqueness of primitive 
parts, we may call a superprimitive part of $H$ \emph{the} primitive part of $H$.
From this, we deduce that GCD-domains are PSP-domains. 
We call $R$ a \emph{GL-domain} if Gauss' Lemma holds for $R$, i.e.\@ if
$f_1$ and $f_2$ are primitive polynomials in $R[t]$, where $t$ is another
variable, then $f_1 f_2$ is a primitive polynomial. It is known that PSP-domains, 
in particular GCD-domains, are GL-domains (see e.g.\@ \cite[Prop.\@ 3.2]{MR2310957}). 

Now suppose that $R$ is a GCD-domain, with fraction field $L$, and $f_1, f_2 \in R[t]$. 
Then there are $a_1, a_2 \in R$ such that $a_1^{-1} f_1$ and $a_2^{-1} f_2$ are the primitive 
parts of $f_1$ and $f_2$ respectively. Let $b := \gcd\{a_1,a_2\}$.
Since $L[t]$ is a Euclidean domain, we can find the greatest common divisor of $f_1$ and 
$f_2$ over $L$ by way of the Euclidean GCD-algorithm. Let $g$ be the greatest common divisor 
of $f_1$ and $f_2$ over $L$. We can choose $g$ such that $b^{-1}g$ is primitive. 

Using the fact that GCD-domains are GL-domains, one can verify that 
$g = \gcd\{f_1,f_2\}$ over $R[t]$. In other words, $R[t]$ is a GCD-domain.
So if $R$ is a GCD-domain, then $R[t]$ is a GCD-domain as well.
By induction on $m$, it follows that $R[y_1,y_2,\ldots,y_m]$ is a GCD-domain and 
hence also a PSP-domain and a GL-domain.

If $R$ is Noetherian, then $R$ is a GCD-domain if and only if $R$ is a PSP-domain
if and only if $R$ is a GL-domain. For instance, $\Z[\sqrt{-5}]$ is not a GCD-domain, 
so $\Z[\sqrt{-5}]$ is neither a PSP-domain nor a GL-domain.

\medskip
In theorem \ref{hmgrk2}, which can be seen as the main theorem, we classifiy
all rational maps $H \in K(x)^n$ for which $\trdeg_K K(tH_1, tH_2, \ldots, tH_n) \le 2$.
To prove Theorem \ref{hmgrk2} and several other results, we will use the following 
classical theorem.

\begin{theorem}[L{\"u}roth-Igusa] \label{luroth}
Assume that $L \supseteq K$ is a subfield of $K(x)$, such that $\trdeg_{K} (L) = 1$.
Then we can write $L = K(p/q)$, such that $p, q \in K[x]$. 
\end{theorem}

Theorem \ref{luroth} is known as L{\"u}roth's theorem, but L{\"u}roth only proved
the case where $K = \C$ and $n = 1$ in \cite{MR1509855}, published in 1876.
L{\"u}roth's result was generalized to arbitrary $n$ in 1887 by Gordan in 
\cite{MR1510417}, and to arbitrary $K$ in 1910 by Steinitz in \cite{MR1580791}. 
Both generalizations were combined in 1951 by Igusa in \cite{MR0048860}. 
See \cite[\S~1.1-1.2]{MR1770638} for a proof of Theorem \ref{luroth} and more 
information about it. In 1953, Samuel published a short and elegant way to get from $n=1$ 
to arbitrary $n$ in the case where $K$ is infinite (see \cite{MR0058251} and the first 
part of the proof of Theorem 3 in \cite[\S~1.2]{MR1770638}). 

Theorem \ref{luroth} can be generalized to $\trdeg_{K} (L) = 2$, but only with additional
requirements. In \cite{MR0099990}, Zariski gives a counterexample over any field $K$ of 
characteristic $c \ge 3$, namely
\begin{equation} 
L = K\Big(x_1^c,x_2^c,-x_1^{c+1}-x_2^{c+1}+\frac{x_1^2+x_2^2}{2}\Big)\mbox{,}
\end{equation}
which indicates that $K(x)$ must be separable over $L$.

The necessity of the requirement that $K$ is algebraically closed, has been 
shown in 1951 by Segre in \cite{MR0042746}. 
The example $\R(\phi_1,\phi_2,\phi_3)$, with
\begin{equation} \label{OjaR}
\phi = \frac{1-x_2^2}{1+x_2^2} \left( \begin{array}{c} 
2 + x_1^2 \\ x_1\, (3 + x_1^2) \\ \sqrt{2} 
\end{array} \right) + \frac{2\,x_2}{1+x_2^2} \left( \begin{array}{c} 
x_2\,(2 + x_1^2) \\ -\sqrt{2} \\ x_1\, (3 + x_1^2) 
\end{array} \right)
\end{equation}
by Ojanguren in \cite{Ojanguren1990} shows this necessity as well. A slight variation 
on \eqref{OjaR} is to take an umbrella rotation offset of $45$ degrees. 
This variation makes $\Q(\phi_1,\phi_2,\phi_3)$ an example as well, because with an 
offset of 45 degrees, the coefficients are just in $\Z$ instead of $\Z[\sqrt{2}]$:
\begin{equation} \label{OjaQ}
\phi = \frac{1-x_2^2}{1+x_2^2} \left( \begin{array}{c} 
2 + 2\,x_1^2 \\ x_1\, (3 + x_1^2) - 1 \\ x_1\, (3 + x_1^2) + 1 
\end{array} \right) + \frac{2\,x_2}{1+x_2^2} \left( \begin{array}{c} 
x_2\,(2 + 2\,x_1^2) \\ -(x_1\, (3 + x_1^2) + 1) \\ x_1\, (3 + x_1^2) - 1 
\end{array} \right)\mbox{.}
\end{equation}
Over $K[\mathrm{i}]$, only two generators are needed, in agreement with theorem 
\ref{castelnuovo} below, namely $\phi_1$ and $\phi_2 + \mathrm{i} \phi_3$.

\begin{theorem}[Castelnuovo-Zariski] \label{castelnuovo}
Assume that $L \supseteq K$ is a subfield of $K(x)$, such that $\trdeg_K (L) = 2$.
Suppose that $K$ is algebraically closed and $K(x)$ is separable over $L$ (both are
fulfilled if $K = \C$). Then we can write $L = K(p/q,p^{*}\!/q^{*})$, 
such that $p, q, p^{*}\!, q^{*} \in K[x]$. 
\end{theorem}

Theorem \ref{castelnuovo} was proved in 1894 for $K = \C$ and $n=2$ by Castelnuovo 
in \cite{MR1510838}, and in 1958 for arbitrary algebraically closed fields by Zariski 
in \cite{MR0099990}, but still for $n = 2$. 
See \cite[\S~1.1]{MR1770638} for more information about Theorem \ref{castelnuovo}.
Since I did not find a proof of Theorem \ref{castelnuovo} for arbitrary $n$ in the 
literature, 
we will complete the proof of this theorem at the end of this section, by reduction 
to the case $n=2$. We do that essentially by applying the above-mentioned reduction
by Samuel, using the following theorem.

\begin{theorem} \label{PSS}
Let $L \subseteq K(x)$. Then $L$ is finitely generated over $K$. Say that
$L = K(\phi_1,\phi_2,\ldots,\phi_s)$. Then
$$
\rk \jac(\phi_1,\phi_2,\ldots,\phi_s) \le \trdeg_K (L)\mbox{,}
$$
and equality holds, if and only if $K(x)$ is separable over $L$.
\end{theorem}

The case where $\trdeg_K (L) = n$ of Theorem \ref{PSS} 
can be extracted from Theorems 10 and 13 in 
\cite{DBLP:conf/mfcs/PandeySS16}, but again, I did not find a full proof in the 
literature. We give a full proof at the end of this section.

Theorem \ref{luroth} cannot be generalized to $\trdeg_{K} (L) = 3$:
a conclusion like $L = K(p/q,p^{*}\!/q^{*},p^{**}\!/q^{**})$ cannot be drawn. 
This was shown in 1972 by Artin and Mumford in \cite{MR0321934}, and by 
Clemens and Griffiths in \cite{MR0302652}, and in 1971 by Iskovskih and Manin in 
\cite{MR0291172}. See \cite[\S~1.1]{MR1770638} for more information 
about those papers. 

But there are no explicit counterexamples in the above three papers. There is however a
claim in \cite{MR0291172} that a counterexample $K(\phi_1,\phi_2,\phi_3,\phi_4)$ 
exists, where $\frac16 \in K$ needs to be assumed, such that the primitive part of $\phi$ 
has degree $24$.

We give two explicit counterexamples $L = K(\phi_1,\phi_2,\phi_3,\phi_4)$ below. 
The first example is over fields $K$ with $\frac16$, and reads as follows:
\begin{equation} \label{Fermat}
\phi = \left( \begin{array}{c} x_1 \\ -x_1 \\ 0 \\ -1 \end{array} \right)
- {\frac {3\,x_1\, ( {x_1}^{3}+2\,x_2-1 ) }{{x_1}^{6}-{x_3}^{3}-3\,{x_2}^{2}+3\,x_2-1}}
\left( \begin{array}{c} 1-x_2 \\ x_2 \\ x_3 \\ -{x_1}^{2} \end{array} \right)\mbox{.}
\end{equation}
\eqref{Fermat} satisfies the following equation:
$$
\phi_1^3 + \phi_2^3 + \phi_3^3 + \phi_4^3 + 1^3 = 0\mbox{.} 
$$
Since $x_1^2 = -(\phi_4 + 1) / (\phi_1 + \phi_2) \in L$ and 
$L(x_1) = K(x_1,x_2,x_3)$, we have $\trdeg_K L = 3$. 
Hence $x_4\, (\phi,1)$ is a parametrization of the Fermat cubic threefold.

The second example is over fields $K$ with $\frac1{22}$, and reads as follows:
\begin{equation} \label{Klein}
\phi = \left( \begin{array}{c} 0 \\ 0 \\ x_1 \\ 0 \end{array} \right) 
- {\frac {x_1\,{x_3}^{2}+2\,x_1\,x_2+1}{{x_1}^{4}\,x_3+x_2\,{x_3}^{2}+{x_2}^{2}}}
\left( \begin{array}{c} -{x_1}^{2} \\ x_3 \\ x_2 \\ 1 \end{array} \right)\mbox{,}
\end{equation}
\eqref{Klein} satisfies the following equation:
$$
\phi_1^2\,\phi_2 + \phi_2^2\,\phi_3 + \phi_3^2\,\phi_4 + \phi_4^2\,1 + 1^2\,\phi_1 = 0\mbox{.}
$$
Since $x_1^2 = -\phi_1/\phi_4 \in L$ and $L(x_1) = K(x_1,x_2,x_3)$, we have $\trdeg_K L = 3$. 
Hence $x_4\, (\phi,1)$ is a parametrization of the Klein cubic threefold.

In the corollary on page 64 of Murre's paper \cite{MR0352089}, the assumption that $K$ 
is algebraically closed can be omitted, if one just assumes that the cubic threefold 
over $K$ is smooth over the algebraic closure of $K$. From this corollary,
one can infer that \eqref{Fermat} and \eqref{Klein} are indeed counterexamples to the 
generalization of Theorem \ref{luroth} to $\trdeg_{K}(L) = 3$. Here, and only here, we 
use the fact that $\frac12 \in K$.

Since $x_3 \in L$ for both \eqref{Fermat} and \eqref{Klein}, both examples are 
counterexamples of the generalization of Theorem \ref{luroth} to 
$\trdeg_{K}(L) = 2$ as well. This is because one can see $K(x_3)$ as the base field, 
which is indeed not algebraically closed, just as $\R$ in \eqref{OjaR} and
$\Q$ in \eqref{OjaQ}.
\eqref{Klein} would not be a counterexample if $\frac1{11} \notin K$, because
$(1,3,3^2,3^3,3^4)$ is a singular point of the Klein cubic threefold in that case.

Examples \eqref{Fermat} and \eqref{Klein} were constructed with the method on the 
beginning of page 19 of \cite{Debarre2014}. In both examples, 
$[K(x_1,x_2,x_3):L] = 2$ and $K(x_1,x_2,x_3)$ is separable over $L$, but the latter
would not hold if $\frac12 \notin K$. Also in both \eqref{Fermat} and \eqref{Klein}, 
$\phi$ is generically $2$-to-$1$, because there exists a rational 
function $a$ such that $\phi(-x_1,a,x_3) = \phi(x_1,x_2,x_3)$. Again, this would not 
hold if $\frac12 \notin K$, because $-x_1 = x_1$ and $a = x_2$ in that case.

As a consequence of Theorems \ref{luroth} and \ref{castelnuovo}, we have the following.

\begin{theorem} \label{lurothlike}
Suppose that $H \in K(x)^m \setminus \{0\}^m$. Take any $i$ such that $H_i \ne 0$.
\begin{enumerate}

\item[\upshape (i)] If $\trdeg_K K(tH) = 1$, then $K(H) = K(H_i)$.

\item[\upshape (ii)] If $\trdeg_K K(tH) = 2$, then $K(H) = K(p/q,H_i)$ for some
$p,q \in K[x]$ which do not depend on the choice of $i$.

\item[\upshape (iii)] If $\trdeg_K K(tH) = 3$ and $K = \C$, then $K(H) = K(p/q,p^{*}\!/q^{*}\!,H_i)$
for some $p,q,p^{*}\!,q^{*} \in K[x]$, which do not depend on the choice of $i$.

\item[\upshape (iv)] If $\trdeg_K K(tH) = 3$, $K$ is algebraically closed, and $K(x)$ is separable
over $K(tH) \cap K(x)$, then $K(H) = K(p/q,p^{*}\!/q^{*}\!,H_i)$
for some $p,q,p^{*}\!,q^{*} \in K[x]$, which do not depend on the choice of $i$. 

\end{enumerate}
\end{theorem}

\begin{proof}
Let $L := K(tH) \cap K(x)$. Notice that
$$
L = K\left(\frac{H_1}{H_i},\frac{H_2}{H_i},\ldots, \frac{H_m}{H_i}\right)\mbox{.}
$$
Since $tH_i$ is transcendental over $L$ and $K(tH) = L(tH_i)$, it follows that
$$
\trdeg_K (L) = \trdeg_K K(tH) - 1\mbox{.}
$$
If $\trdeg_K K(tH) = 1$, then $L = K$ because $K$ is algebraically closed in $K(x)$.
If $\trdeg_K K(tH) = 2$, then $L = K(p/q)$ with $p,q$ as in Theorem \ref{luroth}.
If $\trdeg_K \allowbreak K(tH) = 3$, then $L = K(p/q,p^{*}\!/q^{*})$ with $p,q,p^{*}\!q^{*}$ 
as in Theorem \ref{castelnuovo}, provided the conditions of Theorem \ref{castelnuovo}
are satisfied. This proves Theorem \ref{lurothlike}.
\end{proof}

In Section 2, we provide a structure formula for rational maps $H \in K(x)^m$, for which
$\trdeg_K K(tH) \le 2$, where $t$ is another indeterminate. The case where $K = \C$
and $H$ is a homogeneous polynomial map was proved in \cite[Th.~2.1]{MR2179727}, using
a reducibility theorem for algebraically closed fields, namely the Bertini-Krull theorem
(see Theorem 37 of \cite[\S~3.3]{MR1770638}). 

We remove the condition that $K$ is algebraically closed by using Theorem \ref{luroth},
which is valid for any $K$. The case where $K$ is algebraically closed of Theorem 
\ref{luroth} can also be obtained from the Bertini-Krull theorem.

In Section 3, we deduce some other results from the tools of Section 2 and 
Theorem \ref{luroth}. One of these results is that in the situation of Theorem
\ref{luroth}, $K(p/q)$ contains a nonconstant polynomial, if and only if there are 
$\lambda, \mu \in K$ such that $\lambda p + \mu q = 1$, and either $p \notin K$ or 
$q \notin K$. Some other results are about the integrality of $p/q$ or another 
generator of $K(p/q)$ over $K$-subalgebras of $K(p/q)$.

In Section 4, we apply the above result of Section 2 to generalize a theorem of
Paul Gordan and Max N{\"o}ther in \cite{MR1509898}, see also \cite[Th.~5.3]{1501.05168}, 
which comes down to the following on account of \cite[Th.~4.1]{MR3177043}. 
Suppose that $H \in \C[x]^n$ is a homogeneous polynomial map over $\C$, such that 
$\trdeg_{\C} \C(H) = 2$. If $\jac H \cdot H = 0$ and $g = \gcd\{H_1,H_2,\ldots,H_n\}$, 
then $\tilde{H} := g^{-1}H$ is a polynomial map which satisfies 
$\jac \tilde{H} \cdot \tilde{H}(y) = 0$. 

We generalize this result to any field $K$ and any rational map $H \in K(x)^n$ for which
$\trdeg_K K(tH) \le 2$. Furthermore, we prove that 
$$
\jac H \cdot \jac H = \tr \jac H \cdot \jac H \; 
\Longleftrightarrow \; \jac \tilde{H} \cdot \tilde{H} = 0\mbox{,}
$$
where $\tr M$ is the trace of a square matrix $M$.

\begin{proof}[Proof of Theorem \ref{castelnuovo}]
Take $\phi_1,\phi_2,\ldots,\phi_s$ as in Theorem \ref{PSS}.
Then $\rk \jac (\phi_1,\phi_2,\allowbreak\ldots,\phi_s) = 2$ because $K(x)$ is
separable over $L$. Assume without loss of generality that 
$$
\rk \jac (\phi_1,\phi_2,\ldots,\phi_s,x_3,x_4,\ldots,x_n) = n\mbox{.}
$$
On account of Theorem \ref{PSS}, 
$$
\trdeg_K (\phi_1,\phi_2,\ldots,\phi_s,x_3,x_4,\ldots,x_n) = n
= \trdeg_K (\phi_1,\phi_2,\ldots,\phi_s) + n - 2\mbox{,}
$$
so $x_3,x_4,\ldots,x_n$ are algebraically independent over $L$, and
$K(x)$ is separable over $L(x_3,x_4,\ldots,x_n)$. 
Hence we deduce from the case $n = 2$ that 
$$
L(x_3,x_4,\ldots,x_n) = K(\eta_1,\eta_2,x_3,x_4,\ldots,x_n)\mbox{,}
$$
where $\eta_1,\eta_2 \in K(x)$.

Notice that $K$ is infinite, because $K$ is algebraically closed.
Now the rest of the proof is similar to Samuel's reduction, as in
the first part of the proof of Theorem 3 in \cite[\S~1.2]{MR1770638}.
\end{proof}

\begin{proof}[Proof of Theorem \ref{PSS}]
The first claim of Theorem \ref{PSS} appears in Theorem 1 in \cite[\S~1.1]{MR1770638}. 
To prove the rest, we show the following three statements.
\begin{itemize}

\item \emph{$\rk \jac (\phi_1,\phi_2,\ldots,\phi_s) \le \trdeg_K (L)$.}

Suppose that $\rk \jac (\phi_1,\phi_2,\ldots,\phi_s) > \trdeg_K (L)$, and let $\bar{K}$
be the algebraic closure of $K$. Let $r = \rk \jac (\phi_1,\phi_2,\ldots,\phi_s)$ and 
assume without loss of generality that $\rk \jac (\phi_1,\phi_2,\ldots,\phi_r) = r$.
Then 
$$
\rk \jac (\phi_1,\phi_2,\ldots,\phi_r) > \trdeg_K (L) 
\ge \trdeg_K \big(K(\phi_1,\phi_2,\ldots,\phi_r)\big)\mbox{.}
$$
It follows that $\phi_1,\phi_2,\ldots,\phi_r$ are algebraically dependent over $K$, and hence
over $\bar{K}$, too.

Let $m = r$ and take $\psi \in \bar{K}[y] \setminus \{0\}$, such that 
$\psi(\phi_1,\phi_2,\ldots,\phi_r) = 0$ and $\psi$ has minimum degree.
By the chain rule,
\begin{align*}
0 &= \jac \big(\psi(\phi_1,\phi_2,\ldots,\phi_r)\big) \\
&= \big(\jac_y\psi\big)\big|_{y=(\phi_1,\phi_2,\ldots,\phi_r)} \cdot 
\jac (\phi_1,\phi_2,\ldots,\phi_r)\mbox{.}
\end{align*}
Since the rows of $\jac (\phi_1,\phi_2,\ldots,\phi_r)$ are independent
over $K(x)$ and hence over $\bar{K}(x)$ as well, we conclude that 
$$
\big(\jac_y\psi\big)\big|_{y=(\phi_1,\phi_2,\ldots,\phi_r)} = 0\mbox{.}
$$
As $\psi$ has minimum degree, $\jac_y\psi = 0$ follows. Consequently,
$\psi \in \bar{K}[y_1^c,y_2^c,\allowbreak \ldots,y_r^c]$ where $c$ is the 
characteristic of $\bar{K}$. Since $\bar{K}$ is a perfect field, 
i.e.\@ $\bar{K}$ is closed under taking $c$-th root, 
$\sqrt[c]{\psi} \in \bar{K}[y]$. This contradicts that $\psi$ has minimum degree.

\item \emph{If $K(x)$ is separable over $L$, then 
$\rk \jac (\phi_1,\phi_2,\ldots,\phi_s) \ge \trdeg_K (L)$.}

Suppose that $K(x)$ is separable over $L$ and that 
$\rk \jac (\phi_1,\phi_2,\ldots,\phi_s) < \trdeg_K (L)$.
Let $m = n - \trdeg_K (L) + s$. Since $K(x)$ is separable over $L$, 
there exists a transcendence basis $\xi_{s+1}, \xi_{s+2}, \ldots, \xi_m$ of $K(x)$ over $L$,
such that $K(x)$ is separable over $L(\xi_{s+1}, \xi_{s+2}, \ldots, \xi_m)$.
Furthermore,
$$
\rk \jac (\phi_1,\phi_2,\ldots,\phi_s,\xi_{s+1}, \xi_{s+2}, \ldots, \xi_m)
< \trdeg_K (L) + m - s = n\mbox{,}
$$
so we can take $i$ such that $\jac x_i = e_i\tp$ is not contained in the row space of 
$\jac (\phi_1,\phi_2,\ldots,\phi_s,\xi_{s+1}, \xi_{s+2}, \ldots, \xi_m)$. 

Take $\psi \in K(y)[t]$, such that 
$\psi(\phi_1,\phi_2,\ldots,\phi_s,\xi_{s+1}, \xi_{s+2}, \ldots, \xi_m)[t]$ is
the minimum polynomial of $x_i$ over $L(\xi_{s+1}, \xi_{s+2}, \ldots, \xi_m)$.
By the chain rule,
\begin{align*}
0 &= \jac \big(\psi(\phi_1,\ldots,\phi_s,\xi_{s+1}, \ldots, \xi_m)(x_i)\big) \\
&= \big(\jac_{y,t}\psi\big)
\big|_{y=(\phi_1,\ldots,\phi_s,\xi_{s+1}, \ldots, \xi_m)}\big|_{t=x_i} \cdot 
\jac (\phi_1,\ldots,\phi_s,\xi_{s+1}, \ldots, \xi_m,x_i)\mbox{.}
\end{align*}
Since the last row of $\jac (\phi_1,\ldots,\phi_s,\xi_{s+1}, \ldots, \xi_m,x_i)$
is independent of the rows above it, we conclude that 
$$
\big({\textstyle\parder{}{t}\psi}\big)
\big|_{y=(\phi_1,\ldots,\phi_s,\xi_{s+1}, \ldots, \xi_m)}\big|_{t=x_i} = 0\mbox{.}
$$
As $\psi(\phi_1,\phi_2,\ldots,\phi_s,\xi_{s+1}, \xi_{s+2}, \ldots, \xi_m)[t]$ is
the minimum polynomial of $x_i$ over $L(\xi_{s+1}, \xi_{s+2}, \ldots, \xi_m)$, 
$\parder{}{t} \psi = 0$. So $\psi \in K(y)[t^c]$, where $c$ is the characteristic of $K$,
and $x_i$ is inseparable over $L(\xi_{s+1}, \xi_{s+2}, \ldots, \xi_m)$. A contradiction.

\item \emph{If $\rk \jac (\phi_1,\phi_2,\ldots,\phi_s) \ge \trdeg_K (L)$, 
then $K(x)$ is separable over $L$.}

Suppose that $\rk \jac (\phi_1,\phi_2,\ldots,\phi_s) \ge \trdeg_K (L)$. 
Let $r = \rk \jac (\phi_1,\phi_2,\ldots,\allowbreak \phi_s)$
and assume without loss of generality that $\rk \jac(\phi_1,\phi_2,\ldots,\phi_r) = r$.
Take $\xi_{r+1}, \xi_{r+2},\ldots,\xi_n \in K(x)$, such that 
$\rk \jac(\phi_1,\phi_2,\ldots,\phi_r,\xi_{r+1}, \xi_{r+2},\allowbreak 
\ldots,\xi_n) = n$.

From the first assertion, it follows that 
$\trdeg_K K(\phi_1,\phi_2,\ldots,\phi_r,\xi_{r+1}, \xi_{r+2},\allowbreak
\ldots,\xi_n) = n$ as well. Consequently, $\phi_1,\phi_2,\ldots,\phi_r$ are
algebraically independent over $K$, so $\trdeg_K (L) \ge r \ge \trdeg_K (L)$.
Therefore, $\phi_1,\phi_2,\ldots,\phi_r$ is a transcendence basis of $L$ over $K$, and
$\xi_{r+1}, \xi_{r+2},\ldots,\xi_n$ are algebraically independent over $L$.

Since $\xi_{r+1}, \xi_{r+2},\ldots,\xi_n$ are algebraically independent over $L$, 
and $K(x)$ is algebraic over $L(\xi_{r+1}, \xi_{r+2},\ldots,\xi_n)$, 
it follows that it suffices to prove that $K(x)$ is separable over
$L(\xi_{r+1}, \xi_{r+2},\ldots,\xi_n)$. On account of a transitivity property
for algebraic separable field extensions, it suffices to show that $K(x)$ is separable over
$K(\phi_1,\phi_2,\ldots,\phi_r,\xi_{r+1}, \xi_{r+2},\ldots,\xi_n)$.

So assume that there exists a $\theta \in K(x)$ which is inseparable over 
$K(\phi_1,\phi_2,\allowbreak\ldots,\phi_r,\xi_{r+1}, \xi_{r+2},\ldots,\xi_n)$. 
Let $m = n$ and take $\psi \in K(y)[t]$, such that 
$$
\psi(\phi_1,\phi_2,\ldots,\phi_s,\xi_{s+1}, \xi_{s+2}, \ldots, \xi_n)[t]
$$
is the minimum polynomial of $\theta$ over $L(\xi_{s+1}, \xi_{s+2}, \ldots, \xi_n)$.
Since $\theta$ is inseparable over $L(\xi_{r+1}, \xi_{r+2},\ldots,\xi_n)$,
it follows that $\parder{}{t} \psi = 0$, so the last entry of
$\big(\jac_{y,t}\psi\big)
\big|_{y=(\phi_1,\ldots,\phi_s,\xi_{s+1}, \ldots, \xi_n)}\big|_{t=\theta}$
is zero. The first $n$ rows of $\jac (\phi_1,\ldots,\phi_s,\allowbreak
\xi_{s+1}, \ldots, \xi_n,\theta)$ are independent over $K$, and by the chain rule,
\begin{align*}
0 &= \jac\big(\psi(\phi_1,\ldots,\phi_s,\xi_{s+1}, \ldots, \xi_n)(\theta)\big) \\
&= \big(\jac_{y,t}\psi\big)
\big|_{y=(\phi_1,\ldots,\phi_s,\xi_{s+1}, \ldots, \xi_n)}\big|_{t=\theta} \cdot 
\jac (\phi_1,\ldots,\phi_s,\xi_{s+1}, \ldots, \xi_n,\theta)\mbox{.}
\end{align*}
Consequently, the other entries of $\big(\jac_{y,t}\psi\big)
\big|_{y=(\phi_1,\ldots,\phi_s,\xi_{s+1}, \ldots, \xi_n}\big|_{t=x_i}$
are zero as well.

So $\psi \in K(y_1^c,y_2^c,\ldots,y_n^c)[t^c]$, 
where $c$ is the characteristic of $K$, and $\sqrt[c]{\psi} \in \bar{K}(y)[t]$, where
$\bar{K}$ is the algebraic closure of $K$. Let $B$ be a basis of $\bar{K}$ as
a vector space over $K$. Then $\bar{K}(y)[t]$ is a free module over $K(y)[t]$ with 
basis $B$. Since 
$\psi(\phi_1,\phi_2,\ldots,\phi_s,\xi_{s+1}, \xi_{s+2}, \ldots, \xi_n)(t) \ne 0$,
there exists a $b \in B$, such that for the coefficient $\psi^{*}$ of $b$ in
$\sqrt[c]{\psi}$, 
$$
\psi^{*}(\phi_1,\phi_2,\ldots,\phi_s,\xi_{s+1}, \xi_{s+2}, \ldots, \xi_n)(t) \ne 0\mbox{.}
$$
Notice that $\deg_t \psi^{*} \le \deg_t \sqrt[c]{\psi} < \deg_t \psi$. 
Since $\bar{K}(x)$ is a vector space over $K(x)$ with basis $B$, it follows that 
$$
\psi^{*}(\phi_1,\phi_2,\ldots,\phi_s,\xi_{s+1}, \xi_{s+2}, \ldots, \xi_n)(\theta) = 0\mbox{.}
$$
As $\psi^{*}(\phi_1,\phi_2,\ldots,\phi_s,\xi_{s+1}, \xi_{s+2}, \ldots, \xi_n)(t) \ne 0$,
we have a contradiction with the definition of $\psi$. \qedhere

\end{itemize}
\end{proof}

\section{A structure theorem for homogeneous polynomial maps of transcendence degree 2}
\label{bertinirk2}

We first prove some general statements about the transcendence degree of $K(tH)$
over $K$. Recall that $t$ is another variable, so $\trdeg_K K(x,t) = n + 1$.

\begin{proposition} \label{ttrdeg}
Assume that $H \in K(x)^m$. Then the following hold.
\begin{enumerate}

\item[\upshape (i)] $\trdeg_K K(tH) = \trdeg_K K(tgH)$ for every nonzero $g \in K(x)$.

\item[\upshape (ii)] $\trdeg_K K(H) \le \trdeg_K K(tH) \le \trdeg_K K(H) + 1$.

\item[\upshape (iii)] $\trdeg_K K(H) = \trdeg_K K(tH)$, if and only if $t$ is algebraically
independent over $K$ of (the components of) $tH$.

\item[\upshape (iv)] If the primitive part $\tilde{H}$ of $H$ is homogeneous and $H$ has a nonzero
component $H_j$ which cannot be expressed as a quotient of two homogeneous polynomials of
the same degree, then $\trdeg_K K(H) = \trdeg_K K(tH)$.

\end{enumerate}
\end{proposition}

\begin{listproof}
\begin{enumerate}
 
\item[(i)] This follows from the fact that both $t$ and $tg$ are algebraically independent 
over $K$ of $H_1, H_2, \ldots, H_m$.

\item[(ii)] By substituting $t = 1$ in $tH$, we see that $f(tH) = 0 \Longrightarrow f(H) = 0$
for all $f \in K[y]$. Hence 
$$
\trdeg_K K(H) \le \trdeg_K K(tH) \le \trdeg_K K(H,t) = \trdeg_K K(H) + 1\mbox{.}
$$

\item[(iii)] This follows from $\trdeg_K K(tH,t) = \trdeg_K K(H,t) = \trdeg_K K(H) + 1$.

\item[(iv)] Suppose that the primitive part $\tilde{H}$ of $H$ is homogeneous of
degree $d$. Take $g,\tilde{g} \in K[x]$, such that $\gcd\{g,\tilde{g}\} = 1$
and $g\tilde{H} = \tilde{g}H$. 

If $g$ is homogeneous and $\tilde{g}$ is homogeneous of degree $\deg g + d$, then 
$H_j = g \tilde{H}_j / \tilde{g}$ is a quotient of two homogeneous polynomials of
the same degree, for any $j$ such that $H_j \ne 0$. So assume the opposite. 
We show that 
$$
\tau(x,t) := \frac{g(tx) \cdot t^d}{\tilde{g}(tx)} \notin K(x)\mbox{.}
$$
Suppose first that both $g$ and $\tilde{g}$ are homogeneous. Then
$$
\tau(x,t) := \frac{g(x) \cdot t^{\deg g + d}}{\tilde{g}(x) \cdot t^{\deg \tilde{g}}} \notin K(x)\mbox{,}
$$
because $\deg \tilde{g} \ne \deg g + d$ by assumption. Suppose next that
either $g$ or $\tilde{g}$ is not homogeneous, and say that $\tilde{g}$ is not homogeneous.
Then $\tilde{g}(tx)$ has an irreducible divisor $r$, for which
$\traildeg_t r < \deg_t r$. From $r|_{t=1} \mid \tilde{g}$ and $\gcd\{g,\tilde{g}\} = 1$,
it follows that $r|_{t=1} \nmid g$. Consequently, $r \nmid g(tx)$. So
$$
r \nmid g(tx) \cdot t^d \cdot f(x)
$$
for every $f \in K[x]$. Hence $\tau(x,t) \notin K(x)$.

Therefore, $\tau(x,t)$ is trancendental over $K(x)$. Since
$$
H(tx) = \frac{g(tx)}{\tilde{g}(tx)} \cdot \tilde{H}(tx) 
= \frac{g(tx) \cdot t^d}{\tilde{g}(tx)} \cdot \tilde{H}(x) 
= \tau(x,t) \cdot \tilde{H}(x)
$$
it follows that
$$
\trdeg_K K(H) = \trdeg_K K\big(H(tx)\big) = \trdeg_K K(t\tilde{H})\mbox{.}
$$
Now the result follows from (i). \qedhere

\end{enumerate}
\end{listproof}

Lemma \ref{hf} below connects polynomial maps $f \in K[y_1]^m$ with homogeneous 
polynomial maps $h \in K[y_1,y_2]^n$. This connection is used in Corollary \ref{hfc}.

\begin{lemma} \label{hf}
The following statements are equivalent:
\begin{enumerate}

\item[\upshape (1)] $h \in K[y_1,y_2]^m \setminus \{0\}^m$ is homogeneous of 
                    degree $s$ and $f = h(y_1,1)$.

\item[\upshape (2)] $f \in K[y_1]^m \setminus \{0\}^m$ has degree at most $s$ 
                    and $h = y_2^s f(y_1/y_2)$.

\end{enumerate}
Furthermore, if {\upshape (1)} or {\upshape (2)} holds, then
$$
g = \tilde{g}(y_1,1) \qquad \mbox{and} \qquad \tilde{g} = y_2^{\deg g} g(y_1/y_2)
$$
give a one--one correspondence between common divisors $g \in K[y_1]$ of 
$f_1,f_2,\ldots,\allowbreak f_m$ and common divisors $\tilde{g} \in K[y_1,y_2]$ 
of $h_1,h_2,\ldots,h_m$ for which $y_2 \nmid \tilde{g}$.
\end{lemma}

\begin{listproof}
\begin{description} 

\item[(1) \imp (2)] Assume (1). Since $h$ is homogenoeus of degree $s$, it follows
that 
$$
h = y_2^s h(y_2^{-1}y) = y_2^s h(y_1/y_2,1) = y_2^s f(y_1/y_2)\mbox{.}
$$
As $h \notin K(y_1,y_2) \setminus K[y_1,y_2]$ we deduce that $\deg f \le s$.

\item[(2) \imp (1)] Assume (2). 
Then $h \in K[y_1,y_2]$ because $\deg f \le s$. Furthermore,
$h$ is homogeous of degree $s$ and $f = 1^s f(y_1/1) = h(y_1,1)$.

\end{description}
Suppose first that $\tilde{g} \in K[y_1,y_2]$ such that $\tilde{g} \mid h_i$ 
for each $i \in \{1,2,\ldots,m\}$ and $y_2 \nmid \tilde{g}$, 
and that $g = \tilde{g}(y_1,1)$.
By substituting $y_2 = 1$, we see that $g \mid f_i$ for each $i$. Since 
$h_i$ is homogeneous of degree $s$ for some $i$, it follows that $\tilde{g}$ is
homogeneous, say of degree $\tilde{s}$. From (1) $\Rightarrow$ (2), we deduce
that $\deg g \le \tilde{s}$ and 
$$
\tilde{g} = y_2^{\tilde{s}} g(y_1/y_2) = 
y_2^{\tilde{s} - \deg g} \cdot y_2^{\deg g} g(y_1/y_2)\mbox{.}
$$
If $\deg g < \tilde{s}$, then it follows from (2) $\Rightarrow$ (1) that
$y_2^{\deg g} g(y_1/y_2) \in K[y_1,y_2]$, which contradicts $y_2 \nmid \tilde{g}$.
So $\deg g = \tilde{s}$ and
$$
\tilde{g} = y_2^{\deg g} g(y_1/y_2)\mbox{.}
$$
This proves the correspondence in one direction.

Suppose next that $g \in K[y_1]$ such that $g \mid f_i$ for each $i \in \{1,2,\ldots,m\}$, 
and that $\tilde{g} = y_2^{\deg g} \bcdot g(y_1/y_2)$. Then for each $i$, there exists an 
$a_i \in K[y_1]$ such that $g \cdot a_i = f_i$. Since 
$\deg a_i = \deg f_i - \deg g \le s - \deg g$, it follows from (2) $\Rightarrow$ (1) that
$$
\tilde{g} = y_2^{\deg g} g(y_1/y_2) \in K[y_1,y_2] 
\qquad \mbox{and} \qquad
y_2^{s-\deg g} a_i(y_1/y_2) \in K[y_1,y_2]
$$
for each $i$, and that $\deg \tilde{g} = \deg g$.
So 
$$
\tilde{g} \mid y_2^{\deg g} g(y_1/y_2) \cdot y_2^{s-\deg g} a_i(y_1/y_2) 
= y_2^s f_i(y_1/y_2) = h_i
$$
for each $i$. If $y_2 \mid \tilde{g}$, then 
$$
y_2^{-1} \tilde{g} \in K[y_1,y_2] 
\qquad \mbox{and} \qquad
g = 1^{-1} \tilde{g}(y_1,1)\mbox{,}
$$
and it follows from (1) $\Rightarrow$ (2) that $\deg g \le \deg \tilde{g} - 1$,
which contradicts $\deg \tilde{g} = \deg g$. So $y_2 \nmid \tilde{g}$.
This proves the correspondence in the other direction.
\end{listproof}

\begin{corollary} \label{hfc}
Let $L$ be an extension field of $K$ and $H \in L^m$, such that $H_i \ne 0$
for some $i \in \{1,2,\ldots,m\}$. 
Suppose that there are $p,q \in L$ such that $H_i^{-1}H \in K(p/q)^m$.

Then there exists an $f \in K[y_1]$, a $g \in L$ and a homogeneous polynomial
$h \in K[y_1,y_2]$, such that $f$ and $h$ are primitive, $\deg f = \deg h$, and
$$
H = H_i \cdot f_i^{-1}(p/q) \cdot f(p/q) = g \cdot h(p,q)\mbox{.}
$$
Furthermore, if $s = \deg h$, then $f$, $h$ and $s$ satisfy Lemma \ref{hf}.
\end{corollary}

\begin{proof}
Since $H_i^{-1}H \in K(p/q)^m$, there exists a primitive $f \in K[y_1]^m$ such that 
$H_i^{-1}H = f_i^{-1}(p/q) \cdot f(p/q)$. Let $s := \deg f$ and take 
$h = y_2^s f(y_1/y_2)$. From (2) $\Rightarrow$ (1) of Lemma \ref{hf},
it follows that $h$ is homogeneous of degree $s$. Furthermore, it follows from
Lemma \ref{hf} that $h$ is primitive.

Take $g := H_i/h_i(p,q)$. Then 
$$
H = \frac{H_i}{f_i(p/q)} \cdot f(p/q) 
= \frac{H_i}{q^s f_i(p,q)} \cdot q^s f(p/q) = g \cdot h(p,q)\mbox{.}
$$
Since $f$, $h$ and $s$ satisfy (2) of Lemma \ref{hf}, the last claim follows as well. 
\end{proof}

Proposition \ref{pf} below is a special case of Theorem \ref{ph}. Both
describe a situation where substitution commutes with taking the greatest common divisor. 

\begin{proposition} \label{pf}
Let $R$ be an integral $K$-domain. Take $f \in K[y_1]^m$ and $p \in R$. Then
\begin{equation} \label{feq}
\gcd\big\{f_1(p),f_2(p),\ldots,f_m(p)\big\} 
= \big({\rm gcd}\big\{f_1,f_2,\ldots,f_m\big\}\big)\big|_{y_1=p}\mbox{.}
\end{equation}
In particular, the left-hand side exists.
\end{proposition}

\begin{proof}
Let $g := \gcd\{f_1,f_2,\ldots,f_m\}$. Then $g \mid f_i$ and 
$g(p) \mid f_i(p)$ for each $i$. So $g(p)$ is a common divisor 
of $f_1(p),f_2(p),\ldots,f_m(p)$. 

Let $r$ be another common divisor of $f_1(p),f_2(p),\ldots,f_m(p)$.
Since $K[y_1]$ is a principal ideal domain, there are $a_i \in K[y_1]$ such 
that 
$$
g = a_1 f_1 + a_2 f_2 + \cdots + a_m f_m\mbox{.}
$$
Hence
$$
r \mid \big(a_1(p) f_1(p) + a_2(p) f_2(p) + \cdots + a_m(p) f_m(p)\big) = g(p)\mbox{.}
$$
So $g(p) = \gcd\big\{f_1(p),f_2(p),\ldots,f_m(p)\big\}$, which is
\eqref{feq}.
\end{proof}

\begin{theorem} \label{ph}
Let $R$ be an integral $K$-domain. Take $h \in K[y_1,y_2]^m$ and $p,q \in R$. 
If $h_i$ is homogeneous or zero for each $i \in \{1,2,\ldots,m\}$ and 
$(p,q)$ is superprimitive in $R$, then
\begin{equation} \label{heq}
\gcd\big\{h_1(p,q),h_2(p,q),\ldots,h_m(p,q)\big\} 
= \big({\rm gcd}\big\{h_1,h_2,\ldots,h_m\big\}\big)\big|_{y_1=p}\big|_{y_2=q}\mbox{.}
\end{equation}
In particular, the left-hand side exists.
\end{theorem}

\begin{proof}
Let $\tilde{g} := \gcd\big\{h_1,h_2,\ldots,h_m\big\}$. Then $\tilde{g} \mid 
h_i$ and $\tilde{g}(p,q) \mid h_i(p,q)$ for each $i$. So $\tilde{g}(p,q)$ 
is a common divisor of $h_1(p,q), h_2(p,q), \ldots, h_m(p,q)$. 

Let $r$ be another common divisor of $h_1(p,q), h_2(p,q), \ldots, h_m(p,q)$. 
We must show that $r \mid \tilde{g}(p,q)$ if $(p,q)$ is superprimitive. 
But we first show something weaker than $r \mid \tilde{g}(p,q)$, namely that there 
exists an $s$ such that 
$$
r \mid q^s \tilde{g}(p,q)\mbox{.}
$$
For that purpose, define $f_i := h_i(y_1,1)$ for each $i$.
Notice that $\tilde{g}$ is homogeneous or zero. From Lemma \ref{hf}, it follows that 
$$
g \mid f_i \Longrightarrow y_2^{\deg g} g(y_1/y_2) \mid h_i
\qquad \mbox{and} \qquad 
y_2^{\deg g} g(y_1/y_2) \mid \tilde{g} \Longrightarrow g \mid \tilde{g}(y_1,1)
$$ 
for each $g \in K[y_1]$ and each $i$. So if we take $g := \gcd\{f_1,f_2,\ldots,f_m\}$, then
we can deduce that $g \mid \tilde{g}(y_1,1)$.

Since $g \mid \tilde{g}(y_1,1)$ and $K[y_1]$ is a principal ideal domain, 
there exist $a_1, a_2, \ldots, \allowbreak a_m \in K[y_1]$ such that 
$$
\tilde{g}(y_1,1) = a_1 f_1 + a_2 f_2 + \cdots + a_m f_m\mbox{.}
$$
Let $s_i := \deg f_i$ for each $i$, and take $s$ such that $s \ge \deg a_i + s_i$ for each $i$. 
From (2) $\Rightarrow$ (1) of Lemma \ref{hf}, it follows that 
$b_i := y_2^{s-s_i} a_i(y_1/y_2) \in K[y_1,y_2]$ for each $i$, and that
\begin{align*}
y_2^s \tilde{g}(y_1/y_2,1) &= y_2^s a_1(y_1/y_2) f_1(y_1/y_2) + \cdots + 
   y_2^s a_m(y_1/y_2) f_m(y_1/y_2) \\
&= y_2^{s_1} b_1(y_1,y_2) f_1(y_1/y_2) + \cdots + y_2^{s_m} b_m(y_1,y_2) f_m(y_1/y_2) \\ 
&= b_1(y_1,y_2) h_1(y_1,y_2) + \cdots + b_m(y_1,y_2) h_m(y_1,y_2)\mbox{.}
\end{align*}
From (1) $\Rightarrow$ (2) of Lemma \ref{hf}, it follows that 
$$
y_2^{s-\deg \tilde{g}} \tilde{g}(y_1,y_2) = y_2^s \tilde{g}(y_1/y_2,1)\mbox{.}
$$
Consequently, 
$$
r \mid \big(b_1(p,q) h_1(p,q) + \cdots + b_m(p,q) b_m(p,q)\big) 
= q^{s-\deg \tilde{g}} \tilde{g}(p,q) \mid q^s \tilde{g}(p,q)\mbox{.}
$$
So there exists an $s$ such that $r \mid q^s \tilde{g}(p,q)$. Similarly,
there exists an $s'$ such that $r \mid p^{s'} \tilde{g}(p,q)$. 

Now suppose that $(p,q)$ is superprimitive and $r \nmid \tilde{g}(p,q)$.
From the superprimitivity of $(p,q)$, we deduce that for any $i,i'$,
$$
\frac{p^{i'}q^{i} \tilde{g}(p,q)}{r} (p,q) \in K[x]^2 
~\Longrightarrow~ \frac{p^{i'}q^{i} \tilde{g}(p,q)}{r} \in K[x]\mbox{,}
$$
which is equivalent to
$$
r \nmid p^{i'}q^{i} \tilde{g}(p,q) ~\Longrightarrow~ \big(r \nmid p^{i'}q^{i+1} \tilde{g}(p,q)
\mbox{ or } r \nmid p^{i'+1}q^{i} \tilde{g}(p,q)\big)\mbox{.}
$$
By induction, it follows that there are $i,i'$ such that $i + i' = s + s' - 1$
and $r \nmid p^{i'}q^{i} \tilde{g}(p,q)$. So either $i \ge s$ or $i' \ge s'$.
This contradicts that $r \mid q^s \tilde{g}(p,q)$ and $r \mid p^{s'} \tilde{g}(p,q)$.
\end{proof}

Using Theorem \ref{ph}, we can determine to what extend the 
generator $p/q$ in L{\"u}roth's theorem is unique.

\begin{proposition} \label{pquniq}
Let $R$ be an integral $K$-domain and $p,q \in R$, such that
$p/q$ is transcendental over $K$. Suppose that $p^{*}\!,q^{*} \in R$.
Then $K(p/q) = K(p^{*}\!/q^{*})$, if and only if there exists a
$T \in \GL_2(K)$ such that
\begin{equation} \label{GL2}
\frac{p^{*}}{q^{*}} = \frac{T_{11}p+T_{12}q}{T_{21}p+T_{22}q}\mbox{,}
\qquad \mbox{where } T = \left( \begin{array}{cc} T_{11} & T_{12} \\
T_{21} & T_{22} \end{array} \right) \in \Mat_2(K)\mbox{,}
\end{equation}
if and only if there exists an $S \in \GL_2(K)$ such that
\begin{equation} \label{GL2i}
\frac{p}{q} = \frac{S_{11}p^{*}+S_{12}q^{*}}{S_{21}p^{*}+S_{22}q^{*}}\mbox{,}
\qquad \mbox{where } S = \left( \begin{array}{cc} S_{11} & S_{12} \\
S_{21} & S_{22} \end{array} \right) \in \Mat_2(K)\mbox{,}
\end{equation}
in which case $ST = TS = cI_2$ for some nonzero $c \in K$.
\end{proposition}

\begin{proof}
Suppose first that either \eqref{GL2} or \eqref{GL2i} holds. Assume without loss of 
generality that \eqref{GL2} is satisfied. Then
$$
\frac{p^{*}}{q^{*}} = \frac{T_{11}p+T_{12}q}{T_{21}p+T_{22}q} = 
\frac{T_{11}p/q+T_{12}}{T_{21}p/q+T_{22}}\mbox{,}
$$
so $K(p/q) \supseteq K(p^{*}\!/q^{*})$. Take any nonzero $c \in K$
and take $S = cT^{-1}$. Then $ST = TS = cI_2$ and
\begin{align*}
\frac{p}{q} = \frac{cp}{cq} 
&= \frac{S_{11}(T_{11}p+T_{12}q) + S_{12}(T_{21}p+T_{22}q)}{%
S_{21}(T_{11}p+T_{12}q) + S_{22}(T_{21}p+T_{22}q)} \\
&= \frac{S_{11}(T_{11}p+T_{12}q)/(T_{21}p+T_{22}q) + S_{12}}{%
S_{21}(T_{11}p+T_{12}q)/(T_{21}p+T_{22}q) + S_{22}} \\
&= \frac{S_{11}(p^{*}\!/q^{*}) + S_{12}}{%
S_{21}(p^{*}\!/q^{*}) + S_{22}}\mbox{.} 
\end{align*}
So $K(p/q) \subseteq K(p^{*}\!/q^{*})$ as well.

Suppose next that $K(p/q) = K(p^{*}\!/q^{*})$. Take $f$, $g$ and $h$ as in Corollary 
\ref{hfc}, for $H = (p^{*}\!,q^{*})$ and $i = 2$. Then 
$$
(p^{*}\!,q^{*}) = (q^{*})^{-1} \cdot f(p/q) = g \cdot h(p,q)
$$
and $h$ is primitive. Since $K[y_1,y_2]$ is a PSP-domain, we see that $h$ is superprimitive.

Take $f^{*}$, $g^{*}$ and $h^{*}$ in a similar manner as above, for $H^{*} = (p,q)$ and 
$i = 2$. Then
$$
(p,q) = q^{-1} \cdot f^{*}(p^{*}\!/q^{*}) = g^{*} \cdot h^{*}(p^{*}\!,q^{*})
$$
and $h^{*}$ is primitive. From Theorem \ref{ph}, it follows that 
$\gcd\{h^{*}_1\big(h(y_1,y_2)\big)),\allowbreak h^{*}_2\big(h(y_1,y_2)\big)\} = 1$.

Let $s = \deg h$ and $s^{*} = \deg h^{*}$.
Notice that $p/q = h^{*}_1\big(h(p,q)\big))/h^{*}_2\big(h(p,q)\big))$, so
$$
q h^{*}_1\big(h(p,q)\big)) - p h^{*}_2\big(h(p,q)\big)) = 0\mbox{.}
$$
If we divide both sides by $q^{ss^{*}+1}$, then we get
$$
h^{*}_1\big(f(p/q)\big)) - p/q \cdot h^{*}_2\big(f(p/q)\big)) = 0\mbox{.}
$$
Since $p/q$ is transcendental over $K$, we deduce that
$$
h^{*}_1\big(f(y_1/y_2)\big)) - y_1/y_2 \cdot h^{*}_2\big(f(y_1/y_2)\big)) = 0\mbox{.}
$$
If we multiply both sides by $y_2^{ss^{*}+1}$, then we get
$$
y_2 h^{*}_1\big(h(y_1,y_2)\big)) = y_1 \cdot h^{*}_2\big(h(y_1,y_2)\big))\mbox{.}
$$
Since $K[y_1,y_2]$ is a unique factorization domain, we deduce that
$h^{*}(h) = c \cdot (y_1,y_2)$ for some nonzero $c \in K[y_1,y_2]$ of degree 
$ss^{*} - 1$. As $\gcd\{h^{*}_1\big(h(y_1,y_2)\big)), \allowbreak
h^{*}_2\big(h(y_1,y_2)\big)\} = 1$, we see that $c \in K$ and $ss^{*} = 1$. 
So $T,S \in \GL_2(K)$ which satisfy \eqref{GL2} and \eqref{GL2i} exist,
namely $T = \jac_{y_1,y_2} h$ and $S = \jac_{y_1,y_2} h^{*}$. Furthermore,
$ST = c I_2$, so $S = cT^{-1}$.
\end{proof}

The following theorem is the main theorem of this section.

\begin{theorem} \label{hmgrk2}
Assume that $H \in K(x)^m$, such that $\trdeg_K K(tH) \le 2$. Then 
there are $g \in K(x)$ and $h \in K[y_1,y_2]^m$, and a pair $(p,q) \in K[x]^2$
such that
\begin{enumerate}

\item[\upshape (i)] $H = g \cdot h(p,q)$ and $g \ne 0$;

\item[\upshape (ii)] $h$ is either homogeneous and primitive or zero;

\item[\upshape (iii)] $(p,q)$ is primitive, but not constant.

\end{enumerate}
In that case,
\begin{enumerate}

\item[\upshape (iv)] $\trdeg_K K(tH) =  \trdeg_K K\big(th(p,q)\big) = \trdeg_K K(th)$;

\item[\upshape (v)] $H \ne 0 \Longleftrightarrow h(p,q)$ is primitive 
$\Leftrightarrow h$ is primitive $\Longleftrightarrow h \ne 0$;

\item[\upshape (vi)] $\trdeg_K K(tH) \le 1 \Longleftrightarrow h(p,q) \in K^m \Longleftrightarrow h \in K^m$;

\item[\upshape (vii)] If $h \ne 0$, then 
\begin{align*}
\deg h(p,q) &= s \cdot \deg (p,q)\mbox{,} \qquad \mbox{and} \\
\traildeg h(p,q) &= s \cdot \traildeg (p,q)\mbox{,}
\end{align*}
where $s := \deg h = \traildeg h$;

\item[\upshape (viii)] If $h \ne 0$, then $h(p,q)$ is homogeneous, if and only if $(p,q)$ is homogeneous;

\item[\upshape (ix)] If $\deg (p,q)$ is minimal, then $K(p/q)$ is algebraically closed in $K(x)$.

\end{enumerate}
\end{theorem}

\begin{proof}
We first show that a $g \in K(x)$, a $h \in K[y_1,y_2]^m$ and a pair $p,q \in K[x]$
as in (i), (ii), (iii) exist.

If $H = 0$, then we can take $h = 0$ and impose anything we like on $g, p, q$. 
So assume that $H \ne 0$. Then $\trdeg_K(tH) \ge 1$, so there exists an $i$ such that 
$H_i \ne 0$.

Notice that the $i$-th component of $t H_i^{-1} H$ is just $t$. 
Hence it follows from (ii) and  (iii) of Proposition \ref{ttrdeg} that 
$\trdeg_K K(H_i^{-1}H) = \trdeg_K K(tH_i^{-1} \bcdot H) - 1$. Consequently, 
$L := K(H_i^{-1}H)$ satisfies
$$
\trdeg_K L = \trdeg_K K(tH_i^{-1}H) - 1 = \trdeg_K K(tH) - 1\mbox{,}
$$
because of (i) of Proposition \ref{ttrdeg}.
If $\trdeg_K K(tH) = 1$, then $\trdeg_K L = \trdeg_K K(tH) - 1 = 0$, so 
$H_i^{-1}H \in K^m$. Hence we can take $g = H_i$ and $h = H_i^{-1}H$
and impose anything we like on $p$ and $q$ again, if $\trdeg_K K(tH) = 1$.

So assume from now on that $\trdeg_K K(tH) = 2$. 
Then $\trdeg_K L = \trdeg_K \allowbreak K(tH) - 1 = 1$.
From Theorem \ref{luroth}, it follows that there are $p, q \in K[x]$ such 
that $\gcd\{p,q\} = 1$ and $L = K(p/q)$. Furthermore, $L \ne K$, so
$(p,q)$ satisfies (iii).

Take $f$, $g$ and $h$ as in Corollary \ref{hfc}. Then $h \ne 0$ and
(i) and (ii) are satisfied. Since $K[x]$ is a PSP domain, it follows that 
$(p,q)$ is superprimitive.
\begin{enumerate}

\item[(iv)] Let $s := \deg h$. On account of Corollary \ref{hfc}, $f$, $h$ and $s$
satisfy Lemma \ref{hf}. From (i) of Proposition \ref{ttrdeg}, it follows that
\begin{align*}
\trdeg_K K(tH) &= \trdeg_K K\big(th(p,q)\big) \\ 
&= \trdeg_K K\big(tq^{-s}h(p,q)\big)\mbox{.}
\end{align*}
Since $p/q$ is transcendental over $K(t)$, we have
\begin{align*}
\trdeg_K K\big(tq^{-s}h(p,q)\big) &= \trdeg_K K\big(th(p/q,1)\big) \\ 
&= \trdeg_K K\big(th(y_1/y_2,1)\big)\mbox{.}
\end{align*}
From (i) of Proposition \ref{ttrdeg}, it follows that
\begin{align*}
\trdeg_K K\big(th(y_1/y_2,1)\big) &= \trdeg_K K\big(ty_2^s h(y_1/y_2,1)\big) \\ 
&= \trdeg_K K(th)\mbox{.}
\end{align*}
So $\trdeg_K K(tH) = \trdeg_K K\big(th(p,q)\big) = \trdeg_K K(th)$.

\item[(v)] From $\trdeg_K K(tH) = \trdeg_K K(th)$, it follows that 
$H \ne 0 \Longleftrightarrow h \ne 0$.
By assumption, $h$ is primitive $\Longleftrightarrow h \ne 0$. From Theorem
\ref{ph}, it follows that $h(p,q)$ is primitive, if and only if $h$ is primitive.

So $H \ne 0 \Longleftrightarrow h \ne 0 \Longleftrightarrow h$ is primitive
$\Longleftrightarrow h(p,q)$ is primitive.

\item[(vi)] If $h \in K^m$, then $h(p,q) \in K^m$. If $h(p,q) \in K^m$, then
$\trdeg_K K\big(th(p,q)\big) \le 1$, and it follows from (i) of Proposition \ref{ttrdeg}
that $\trdeg_K K(tH) \le 1$. So it remains to show that
$\trdeg_K K(tH) \le 1 \Longrightarrow h \in K^m$.

So assume that $\trdeg_K K(tH) \le 1$. If $H = 0$, then $h = 0 \in K^m$ on account of (v),
so assume that $H \ne 0$. Then (iv) tells us that $\trdeg_K K(t h) = \trdeg_K K(t H) = 1$.
From (ii), (iii) and (i) of Proposition \ref{ttrdeg}, it follows that
$$
\trdeg_K K(h_i^{-1}h) < \trdeg_K K(th_i^{-1}h) = \trdeg_K K(t h) = 1\mbox{,}
$$
so $\trdeg_K K(h_i^{-1}h) = 0$ and $h_i^{-1}h \in K^m$. As $h$ is primitive 
on account of (v), we conclude that $h \in K^m$.

\item[(vii)] Suppose that $h \ne 0$. Then $h$ is primitive by assumption, so
$\gcd\{h_1,h_2,\ldots,\allowbreak h_m\}$ does not have a linear factor. 
Hence the claims follow from Lemma \ref{degs} below.

\item[(viii)] Suppose that $h \ne 0$. Then it follows from (vii) that
$$
\traildeg h(p,q) = \deg h(p,q) \Leftrightarrow \traildeg (p,q) = \deg (p,q)\mbox{.}
$$
So $h(p,q)$ is homogeneous, if and only if $(p,q)$ is homogeneous.

\item[(ix)] From Theorem \ref{luroth}, it follows that there are $p^{*}\!,q^{*} \in K[x]$ 
such that $K(p^{*}\!/q^{*})$ is the algebraic closure of $K(p/q)$. Assume without loss
of generality that $(p^{*}\!,q^{*})$ is primitive. Since $K[x]$ is a PSP-domain,
we see that $(p^{*}\!,q^{*})$ is superprimitive.

Take $f^{*}$, $g^{*}$ and $h^{*}$ as in the proof of Theorem \ref{pquniq}, i.e.\@
as in Corollary \ref{hfc}, but starred, for $H^{*} = (p,q)$ and $i = 2$.
From Proposition \ref{ph}, it follows that $h^{*}(p^{*}\!,q^{*})$
is primitive. Since $h^{*}(p^{*}\!,q^{*})$ is dependent
over $K(x)$ of $(p,q)$, it follows that $(p,q)$ is dependent over $K[x]$ of
$h^{*}(p^{*}\!,q^{*})$.

Since $\deg (p,q)$ is minimal, $(p,q)$ is dependent over $K$ of
$h^{*}(p^{*}\!,q^{*})$. Furthermore, it follows from (vii) that $\deg h^{*} = 1$ 
and $\deg (p,q) = \deg (p^{*}\!,q^{*})$. Consequently, $\deg (p^{*}\!,q^{*})$
is minimal. Since $p/q \notin K$ and $K$ is algebraically closed in $K(x)$, it 
follows from Theorem \ref{pquniq} that $K(p/q) = K(p^{*}\!/q^{*})$. So $K(p/q)$
is algebraically closed in $K(x)$. \qedhere

\end{enumerate}
\end{proof}

\begin{lemma} \label{degs}
Suppose that $h \in K[y_1,y_2]^m$ is homogeneous and $p, q \in K[x]$, not both zero.
If $\gcd\{h_1,h_2,\ldots,h_m\}$ does not have a linear factor, then $h(p,q) \ne 0$, 
and
\begin{align*}
\deg h(p,q) &= s \cdot \deg (p,q)\mbox{,} \qquad \mbox{and} \\
\traildeg h(p,q) &= s \cdot \traildeg (p,q)\mbox{,}
\end{align*}
where $s := \deg h = \traildeg h$.
\end{lemma}

\begin{proof}
Assume without loss of generality that $\deg p \ge \deg q$. Then $p \ne 0$.
By replacing $q$ by a linear combination of $p$ and $q$ and adapting 
$h$ accordingly, we can ensure that either $q = 0$ or that the leading homogeneous 
parts (homogeneous parts of maximum degree) of $p$ and $q$ are linearly independent 
over $K$. 
By subsequently replacing $p$ by a linear combination of $p$ and $q$, and adapting 
$h$ accordingly, we can ensure that either $q = 0$ or that the trailing homogeneous 
parts (homogeneous parts of minimum degree) of $p$ and $q$ are linearly independent 
over $K$ as well.

We first show that $\deg h(p,q) = s \cdot \deg (p,q)$.
Let $(\bar{p},\bar{q})$ be the leading homogeneous part of $(p,q)$ 
(so $\deg (\bar{p},\bar{q}) = \deg (p,q)$).
Since $h$ is homogeneous, it follows that $h(\bar{p},\bar{q})$ is either the leading 
homogeneous part of $h(p,q)$ or zero. By assumption, the exists an $i$ such that
$y_2 \nmid h_i$. Furthermore, if $h_i(\bar{p},\bar{q}) \ne 0$, then 
$$
\deg h(p,q) = \deg h_i \cdot \deg(\bar{p},\bar{q}) = s \cdot \deg (p,q)\mbox{.}
$$
Therefore, assume that $h_i(\bar{p},\bar{q}) = 0$. We will derive a contradiction,
by which we additionally obtain $h(p,q) \ne 0$.

If $\bar{q} = 0$, then $\bar{p} = 0$ as well, because $h_i(\bar{p},\bar{q}) = 0$
and $y_2 \nmid h_i$, which is a contradiction. So $\bar{q} \ne 0$. 
As $\deg p \ge \deg q$, $\bar{p} \ne 0$ as well. Since $h_i$ is
homogeneous, we deduce from $h_i(\bar{p},\bar{q}) = 0$ that 
$h_i(\bar{p}/\bar{q},1) = 0$. So $\bar{p}/\bar{q}$ is algebraic over $K$. 
As $K$ is algebraically closed in $K(x)$, we see that $\bar{p}/\bar{q} \in K$. 
This contradicts the fact that that the leading homogeneous parts of $p$ and $q$ 
are linearly independent over $K$.

We next show that $\traildeg h(p,q) = \traildeg h \cdot \traildeg (p,q)$. 
Let $(\tilde{p},\tilde{q})$ be the trailing homogeneous part of $(p,q)$
(so $\traildeg (\tilde{p},\tilde{q}) = \traildeg (p,q)$).
Since $h$ is homogeneous, it follows that $h(\tilde{p},\tilde{q})$ is the trailing 
homogeneous part of $h(p,q)$ or zero. If $\traildeg p \le \traildeg q$,
then take $j$ such that $y_2 \nmid h_j$. Otherwise, take $j$ such that $y_1 \nmid h_j$. 
If $h_j(\tilde{p},\tilde{q}) \ne 0$, then 
$$
\traildeg h(p,q) = \deg h_j \cdot \deg(\tilde{p},\tilde{q}) = s \cdot \traildeg (p,q)\mbox{.}
$$
Therefore, assume that $h_j(\tilde{p},\tilde{q}) = 0$. Then we can obtain a 
contradiction in a similar manner as in the case $h_i(\bar{p},\bar{q}) = 0$ above.
\end{proof}

\section{Some further results that follow from L{\"u}roth's theorem}

In Theorem 4 in \cite[\S~1.2]{MR1770638}, it is shown that $K(p/q)$ is generated 
by a polynomial if $K(p/q)$ contains a nonconstant polynomial.
Inspection of the proof of this theorem reveals that even $K(p/q) \in \{K(p),K(q)\}$
if $K(p/q)$ contains a nonconstant polynomial and $\gcd\{p,q\} = 1$. This result
can also be obtained from (2) $\Rightarrow$ (3), (4) of Theorem \ref{enother} below.
So Theorem 4 in \cite[\S~1.2]{MR1770638} can be seen as a special case of
Theorem \ref{enother} below.

\begin{theorem} \label{enother}
Let $p, q \in K[x]$ which are not both constant, such that $\gcd\{p,q\} = 1$.
Then for
\begin{enumerate}

\item[\upshape(1)] $p$ and $q$ are algebraically dependent over $K$ and
$K(p/q)$ is algebraically closed in $K(x)$;

\item[\upshape(2)] $K(p/q) \cap K[x] \ne K$;

\item[\upshape(3)]
there are $\lambda, \mu \in K$ such that $\lambda p + \mu q = 1$;

\item[\upshape(4)] $K(p/q) = K(p,q)$;

\end{enumerate}
we have $(1) \Rightarrow (2) \Rightarrow (3) \Rightarrow (4)$.
\end{theorem}

\begin{proof}
Actually, we have $(1) \Rightarrow (2) \Leftrightarrow (3) \Leftrightarrow (4)$. 
But $(4) \Rightarrow (2)$ is trivial. Hence $(1) \Rightarrow (2) \Rightarrow (3)
\Rightarrow (4)$ suffices.
\begin{description}

\item[(1) \imp (2)]
Suppose that $p$ and $q$ are algebraically dependent over $K$. Then
$\trdeg_K K(tp,tq,t) = 2$. From Theorem \ref{hmgrk2}, it follows that 
there exist a $g^{*} \in K(x)$, a homogeneous $h^{*} \in K[y_1,y_2]^3$ and a pair 
$(p^{*}\!,q^{*}) \in K[x]^2$, such that
$$
(p,q,1) = g^{*} h^{*}(p^{*}\!,q^{*})\mbox{.}
$$
Consequently,
\begin{align*}
(p/q,p,q) &= \Big(
\frac{h_1^{*}(p^{*}\!,q^{*})}{h_2^{*}(p^{*}\!,q^{*})},
\frac{h_1^{*}(p^{*}\!,q^{*})}{h_3^{*}(p^{*}\!,q^{*})},
\frac{h_2^{*}(p^{*}\!,q^{*})}{h_3^{*}(p^{*}\!,q^{*})}
\Big) \\ &= \Big(
\frac{h_1^{*}(p^{*}\!/q^{*},1)}{h_2^{*}(p^{*}\!/q^{*},1)},
\frac{h_1^{*}(p^{*}\!/q^{*},1)}{h_3^{*}(p^{*}\!/q^{*},1)},
\frac{h_2^{*}(p^{*}\!/q^{*},1)}{h_3^{*}(p^{*}\!/q^{*},1)}
\Big) \in K(p^{*}\!/q^{*})^3\mbox{.}
\end{align*}
Assume that $K(p/q)$ is algebraically closed in $K(x)$.
From $p/q \in K(p^{*}\!/q^{*})$, it follows that $p^{*}\!/q^{*}$ is algebraic 
over $K(p/q)$. Consequently,
$$
p,q \in K(p^{*}\!/q^{*}) = K(p/q)\mbox{.}
$$
Hence $K(p,q) \cap K[x] \ne K$.

\item[(2) \imp (3)] Suppose that $K(p/q)$ contains a nonconstant
polynomial $r$. Then there exist $f_1,f_2 \in K[y_1]$ such that 
$\gcd\{f_1,f_2\} = 1$ and $r = f_1(p/q)/\allowbreak f_2(p/q)$. Let $s := 
\max\{\deg f_1, \deg f_2\}$. 

From (2) $\Rightarrow$ (1) of Lemma \ref{hf}, it follows that there exist
$h_1, h_2 \in K[y_1,y_2]$, both homogeneous of degree $s$, such that
$$
r = \frac{h_1(p/q,1)}{h_2(p/q,1)} =  \frac{q^s h_1(p/q,1)}{q^s h_2(p/q,1)} 
  = \frac{h_1(p,q)}{h_2(p,q)}\mbox{.}
$$
Furthermore, we deduce from Lemma \ref{hf} that $\gcd\{h_1,h_2\} = 1$.
On account of Theorem \ref{ph}, $\gcd\{h_1(p,q),h_2(p,q)\} = 1$. As 
$r \in K[x]$, we conclude that $h_2(p,q)$ is a unit in $K[x]$, i.e.
$h_2(p,q) \in K$. 

If we take for $h_3$ an irreducible factor of $h_2$, then 
$h_3(p,q) \in K$ as well. Assume without loss of generality that 
$h_3(p,q) = 1$. From Lemma \ref{degs}, it follows that $h_3$ has a linear factor,
so $h_3$ is homogeneous of degree $1$. Say that $h_3 = \lambda y_1 + \mu y_2$,
where $\lambda, \mu \in K$. Then $1 = \lambda p + \mu q$.

\item[(3) \imp (4)] Suppose that there are $\lambda, \mu \in K$ such that 
$1 = \lambda p + \mu q$. Take $r^{*} \in \{p,q\}$ such that $r^{*} \notin K$. 
From Proposition \ref{pquniq}, it follows that $K(p/q) = K(r^{*}/1) = K(p,q)$. 
\qedhere

\end{description}
\end{proof}

As a corollary of Theorem 4 in \cite[\S~1.2]{MR1770638} or (2) $\Rightarrow$ (3), (4) 
of Theorem \ref{enother}, we have the following result.

\begin{corollary} \label{gordan}
Let $R$ be a $K$-subalgebra of $K[x]$, with fraction field $L$, such that 
$\trdeg_{K} L = 1$. Then the following holds.
\begin{enumerate}

\item[\upshape(i)] There exists a $p \in K[x]$ such that $L = K(p)$.

\item[\upshape(ii)] For any $p \in K[x]$ such that $L = K(p)$, 
$$
R \subseteq K[p] = K(p) \cap K[x]\mbox{,}
$$
and $K[p]$ is the integral closure of $R$ (in its fraction field $L$). 

\end{enumerate}
\end{corollary}

\begin{proof} 
From Theorem \ref{luroth}, it follows that there exist $p,q \in K[x]$ 
such that $L = K(p/q)$. Since $\trdeg_{K} L = 1$, it follows that 
$K \subsetneq R \subseteq K[x]$, so $L \supseteq R$ contains a nonconstant 
polynomial. From (2) $\Rightarrow$ (3), (4) of Theorem \ref{enother}, 
we deduce that $L \in \{K(p),K(q)\}$. So we can choose $q = 1$ above.

From Lemma \ref{rkp} below, it follows that
$$
K[p] = K(p) \cap K[x] = L \cap K[x] \supseteq R\mbox{.}
$$
Since $R$ contains a nonconstant polynomial, it follows that $p$ is integral
over $R$. From $K[p] = L \cap K[x]$ and the fact that $K[x]$ is integrally closed
in $K(x)$, we deduce that $K[p]$ is integrally closed in $L$.
So $K[p]$ is the integral closure of $R$ in $L$.
\end{proof}

The ring $R = K[x_1^2,x_1^3]$ is not of the form $K[p]$ for any
$p \in K[x]$, which corresponds to the fact that $R$ is not integrally closed.

In \cite[Lm.~2]{MR0913158}, the author deduces from Theorem 4 in 
\cite[\S~1.2]{MR1770638} that $R \subseteq K[p]$ for some $p \in K[x]$ 
in the situation of Corollary \ref{gordan}, which is a bit less specific.
Zaks in \cite{MR0280471} and Eakin in \cite{MR0289498} prove that $R = K[p]$
for some $p \in K[x]$ if $R$ is integrally closed (in $L$) in the situation of 
Corollary \ref{gordan}, which is almost as specific as Corollary \ref{gordan} itself.
But the proofs of Zaks and Eakin differ from that of \cite[Lm.~2]{MR0913158} and 
Corollary \ref{gordan}.

\begin{lemma} \label{rkp}
Let $p \in K[x]$. Then $K[p] = K(p) \cap K[x]$.
\end{lemma}

\begin{proof}
The inclusion $K[p] \subseteq K(p) \cap K[x]$ is trivial, so assume 
that $r \in K(p) \cap K[x]$. Then there are $f_1,f_2 \in K[y_1]$ such that 
$r = f_1(p)/f_2(p)$ and $\gcd\{f_1,f_2\} = 1$. From Proposition \ref{pf}, it follows 
that $\gcd\{f_1(p),f_2(p)\} = 1$, so $f_2(p)$ is a unit in $K[x]$.
Hence $f_2(p) \in K$ and $r \in K[p]$.
\end{proof}

If $f \in K[y_1] \setminus K$ and $p,q \in K[x]$, then $p/q$ is integral over $f_1(p/q)$.
But if $f \in K(y_1) \setminus K$ and $p,q \in K[x]$, then $p/q$ does not need to be integral 
over $f(p/q)$. The following theorem describes in which case $p/q$ is integral 
over $f(p/q)$. 

\begin{theorem} \label{f2th}
Suppose that $g = f_1(p/q)/f_2(p/q)$ for certain $p,q \in K[x]$ and
$f_1, f_2 \in K[y_1] \setminus K$. If $p/q \notin K$, then the following holds.
\begin{enumerate}
 
\item[\upshape (i)] $p/q$ is integral over $K[g]$, if and only if $\deg f_1 > \deg f_2$.

\item[\upshape (ii)] There are $p^{*}\!,q^{*} \in K[x]$ for which $K(p^{*}\!/q^{*}) = K(p/q)$,
such that $p^{*}\!/q^{*}$ is integral over $K[g]$ if and only if either $\deg f_1 > \deg f_2$ or 
$f_2$ has a root in $K$ of which the multiplicity is larger than that of $f_1$.

\end{enumerate}
\end{theorem}

\begin{proof}
Suppose that $p/q \notin K$. Then $p/q$ is transcendental over $K$.
\begin{enumerate}

\item[(i)] Notice that $f_1(p/q) - g f_2(p/q) = 0$. Assume without loss of generality that
$f_1$ is monic. If $\deg f_1 > \deg f_2$, then $f_1 - g f_2$ is monic as well, so that
$p/q$ is integral over $K[g]$. Hence suppose that $\deg f_1 \le \deg f_2$. Then the leading 
coefficient of $f_1 - g f_2$ has degree one in $g$ and every other coefficient of 
$f_1 - g f_2$ has degree at most one in $g$. 

Let $m \in K[g][y_1]$ be a primitive polynomial over $K[g]$ of minimum degree, such that 
$m(p/q) = 0$. Then $m \mid (f_1 - g f_2)$ over $K(g)$, because otherwise we can decrease the degree
of $m$, namely by replacing it by the primitive part of the remainder of the division over $K(g)$
of $f_1 - g f_2$ by $m$. 

Since $K[g]$ is a GCD-domain, we can take the primitive part of $(f_1 - g f_2)/m$, 
which we call $\tilde{m}$. Using the fact that that $K[g]$ is a GL-domain and a PSP-domain, 
we deduce that the product $m \cdot \tilde{m}$ is primitive and superprimitive respectively. 
It follows that $m \cdot \tilde{m}$ is the primitive part of $f_1 - g f_2$. Hence
$m \mid (f_1 - g f_2)$ over $K[g]$ as well. 

Since every coefficient of $f_1 - g f_2$ has degree at most one in $g$, we can deduce that
either $m \in K[y_1]$ or $(f_1 - g f_2)/m \in K[y_1]$. As $p/q$ is transcendent over $K$,
it follows that $(f_1 - g f_2)/m \in K[y_1]$. Consequently, the leading coefficient of 
$m$ is not contained in $K$. 

If $p/q$ would be integral over $K[g]$ then it follows
in a similar manner as $m \mid (f_1 - g f_2)$ over $K[g]$ that $m$ is a divisor over
$K[g]$ of a monic polynomial over $K[g]$, which is impossible because the 
leading coefficient of $m$ is not contained in $K$. So $p/q$ is not integral over $K[g]$.

\item[(ii)] From Proposition \ref{pquniq}, it follows that we may assume that
either $q^{*} = q$ and $p^{*} = p + \epsilon q$, or $q^{*} = p - \theta q$ and
$p^{*} = q + \epsilon(p - \theta q)$, where $\epsilon, \theta \in K$. 
Hence (ii) follows from Lemma \ref{pqtrans} below. \qedhere

\end{enumerate}
\end{proof}

\begin{lemma} \label{pqtrans}
Assume that $p, q \in K[x]$ and $g \in K(p/q)$, and say that 
$g = f_1(p/q)/\allowbreak f_2(p/q)$, where $f_1,f_2\in K[y_1]$. 
Define $s := \max\{\deg f_1,\deg f_2\}$.

If $\epsilon, \theta \in K$, then the following holds.
\begin{enumerate}
 
\item[\upshape (i)] Suppose that $q^{*} = q$ and $p^{*} = p + \epsilon q$.
Define
\begin{align*}
f^{*}_1 &:= f_1(y_1-\epsilon) & f^{*}_2 &:= f_2(y_1-\epsilon)\mbox{.}
\end{align*}
Then 
\begin{equation} \label{fraceq}
\frac{f^{*}_1(p^{*}\!/q^{*})}{f^{*}_2(p^{*}\!/q^{*})} = \frac{f_1(p/q)}{f_2(p/q)}\mbox{,}
\end{equation}
and $\deg f^{*}_1 > \deg f^{*}_2 \Longleftrightarrow \deg f_1 > \deg f_2$.

\item[\upshape (ii)] Suppose that $q^{*} = p - \theta q$ and $p^{*} = q + \epsilon(p - \theta q)$.
Define
\begin{align*}
f^{*}_1 &:= (y_1-\epsilon)^s \cdot f_1\Big(\frac{1}{y_1-\epsilon}+\theta\Big) &
f^{*}_2 &:= (y_1-\epsilon)^s \cdot f_2\Big(\frac{1}{y_1-\epsilon}+\theta\Big)\mbox{.}
\end{align*}
Then $f^{*}_1, f^{*}_2 \in K[y_1]$, 
\begin{equation*} \tag{\ref{fraceq}}
\frac{f^{*}_1(p^{*}\!/q^{*})}{f^{*}_2(p^{*}\!/q^{*})} = \frac{f_1(p/q)}{f_2(p/q)}\mbox{,}
\end{equation*}
and $\deg f^{*}_1 > \deg f^{*}_2 \Longleftrightarrow$ $\theta$ is a root of $f_2$
from which the multiplicity exceeds that of $f_1$.

\end{enumerate}
\end{lemma}

\begin{proof}
Suppose that $\epsilon, \theta \in K$. 
\begin{enumerate}
 
\item[\upshape (i)] For every $i \le 2$,
$$
f^{*}_i(p^{*}\!/q^{*}) = f_i\big((p+\epsilon q)/q-\epsilon\big) = f_i(p/q)\mbox{,}
$$
from which \eqref{fraceq} follows. The last claim follows from 
$\deg f^{*}_1 = \deg f_1$ and $\deg f^{*}_2 = \deg f_2$.

\item[\upshape (ii)] From (i), it follows that we may assume that $\epsilon = 0$.
So $q^{*} = p - \theta q$, $p^{*} = q$, $f^{*}_1 = y_1^s \cdot f_1(y_1^{-1} + \theta)$ 
and $f^{*}_2 = y_1^s \cdot f_2(y_1^{-1} + \theta)$. Take $i \le 2$ as arbitrary. 
From $\deg f_i \le s$, we deduce that $f^{*}_i \in K[y_1]$. Furthermore
$$
f^{*}_i\Big(\frac{p^{*}}{q^{*}}\Big) = \Big(\frac{q}{p - \theta q}\Big)^s 
f_i\Big(\frac{p - \theta q}{q} + \theta\Big) = \Big(\frac{q}{p - \theta q}\Big)^s 
f_i\Big(\frac{p}{q}\Big)\mbox{.}
$$
So $f^{*}_1, f^{*}_2 \in K[y_1]$ and \eqref{fraceq} is satisfied. 

It remains to prove the last claim of (ii).
Notice that $f_i(y_1 + \theta)$ and $f^{*}_i = y_1^s \cdot f_i(y_1^{-1} + \theta)$
can be changed into each other by reversing the order of the coefficients of $y_1^0, y_1^1, \ldots,
y_1^s$. Hence 
$$
s - \deg f^{*}_i = \traildeg f_i(y_1 + \theta)\mbox{,}
$$
which is the multiplicity of $0$ as a root of $f_i(y_1 + \theta)$ and of $\theta$ as a root of $f_i$.
So $\deg f^{*}_1 > \deg f^{*}_2 \Longleftrightarrow 
\traildeg f_1(y_1 + \theta) < \traildeg f_2(y_1 + \theta)$, and the last claim of (ii)
follows. \qedhere

\end{enumerate}
\end{proof}

\begin{corollary} \label{bavula}
Let $p,q \in K[x]$. Suppose that $R \ne K$ is a $K$-subalgebra of $K(p/q)$ which is integrally 
closed in $K(p/q)$, such that either $K$ is algebraically closed, or $R$ contains a nonconstant 
polynomial.

Then there exist $p^{*}\!,q^{*} \in K[x]$ such that $K(p^{*}\!/q^{*}) = K(p/q)$ 
and $p^{*}\!/q^{*} \in R$, where $q^{*} = 1$ if $R$ contains a nonconstant polynomial.
\end{corollary}

\begin{proof}
Since $R \ne K$, we can take $g \in R \setminus K$, and we have $p/q \notin K$.

First suppose that $K$ is algebraically closed.
Write $g = f_1(p/q)/f_2(p/q)$ such that $\gcd\{f_1,f_2\} = 1$.
Since $K$ is algebraically closed, either $\deg f_1 > \deg f_2$ or $f_2$ has a root
from which the multiplicity exceeds that of $f_1$. From Theorem \ref{f2th}, it
follows that there exist $p^{*}\!,q^{*} \in K[x]$, such that $K(p^{*}\!/q^{*}) = K(p/q)$ and
$p^{*}\!/q^{*}$ is integral over $R$. Hence $p^{*}\!/q^{*} \in R$ by assumption on $R$.

Next suppose next that $R$ contains a nonconstant polynomial $r$. Take $p^{*}\!,q^{*}$
such that $K(p/q) = K(p^{*}\!/q^{*})$ and $\gcd\{p^{*}\!,q^{*}\} = 1$. 
From (2) $\Rightarrow$ (3), (4) of Theorem \ref{enother}, we deduce that 
$K(p/q) \in \{K(p^{*}),K(q^{*})\}$, and we say that $K(p/q) = K(p^{*})$. 
Then we can take $q^{*} = 1$.

Since $r \in K(p^{*}) \cap K[x]$, it follows from Lemma \ref{rkp} that $r \in K[p^{*}]$.
As $r$ is nonconstant, $p^{*} \in K(p/q)$ is integral over $K[r] \subseteq R$. 
Hence $p^{*}\!/q^{*} = p^{*} \in R$ by assumption on $R$.
\end{proof}

The case where $K$ is an algebraically closed field of characteristic zero of 
Corollary \ref{bavula} appears as Corollary 2.2 in \cite{MR2134288}.

Notice that the conditions on $f_1$ and $f_2$ in Theorem \ref{f2th} can be viewed
as $f_2$ having a root in the projective line $K \cup \{\infty\}$
over $K$ with larger multiplicity than $f_1$. 
The formulation of the following lemma was inspired by that idea.

\begin{lemma} \label{vlem}
Suppose that $f_1,f_2,f_1^{*}\!,f_2^{*} \in K[y_1]$, such that $f_2f_2^{*} \ne 0$. 
For every $g\in K[y_1]$, write $v_\theta(g)$ for the multiplicity of $\theta$ as a 
root of $g$ if $\theta \in K$, and define $v_{\infty}(g) := -\deg g$. Then
$$
v_{\theta}(f_1/f_2) := v_\theta(f_1) - v_\theta(f_2)
$$
is well-defined, and
\begin{align*}
\big(v_\theta(f_1) - v_\theta(f_2)\big) + \big(v_\theta(f^{*}_1) - v_\theta(f^{*}_2)\big) 
&= v_\theta(f_1f_1^{*}) - v_\theta(f_2f_2^{*}) \\
\min\big\{v_\theta(f_1) - v_\theta(f_2), v_\theta(f^{*}_1) - v_\theta(f^{*}_2)\big\}
&\le v_\theta(f_1f_2^{*}+f^{*}_1f_2) - v_\theta(f_2f_2^{*})
\end{align*}
for all $\theta \in K \cup \{\infty\}$.
\end{lemma}

\begin{proof}
Take $\theta \in K \cup \{\infty\}$ as arbitrary. The equality is easy to prove, and 
$$
f_1/f_2 = f_2^{*}/f_1^{*} ~\Longrightarrow~ 
v_\theta(f_1) - v_\theta(f_2) = v_\theta(f^{*}_2) - v_\theta(f^{*}_1)
$$
follows from it, so $v_{\theta}(f_1/f_2)$ is well-defined.
The inequality follows from the following.
\begin{align*}
v_\theta(f_1) - v_\theta(f_2) &= v_\theta(f_1f_2^{*}) - v_\theta(f_2f_2^{*}) \\
v_\theta(f^{*}_1) - v_\theta(f^{*}_2) &= v_\theta(f^{*}_1f_2) - v_\theta(f_2f_2^{*}) \\
\min\big\{v_\theta(f_1f_2^{*}),v_\theta(f^{*}_1f_2)\} &\le v_\theta(f_1f_2^{*}+f^{*}_1f_2)
\mbox{.} \qedhere
\end{align*}
\end{proof}

\begin{theorem} \label{Ggth}
Let $p, q \in K[x]$ and let $G$ be a subset of $K(p/q)$. Then $p/q$ is integral over
$K[G]$, if and only if there exist $g \in G$ such that $p/q$ is integral over $K[g]$.
\end{theorem}

\begin{proof}
This follows from (i) of Theorem \ref{f2th} and Lemma \ref{vlem} with $\theta = \infty$.
\end{proof}

One can wonder whether the condition that either $K$ is algebraically closed,
or $R$ contains a nonconstant polynomial, is necessary in Corollary \ref{bavula}. 
The answer is affirmative, as the following example makes clear.

\begin{example}
Take for $R$ the integral closure of $\R[1/(x_1^2+1),x_1/(x_1^2+1)]$ in its 
fraction field $\R(x_1)$. One can show that $R = \R[1/(x_1^2+1),x_1/(x_1^2+1)]$, but there is 
no need to do that to show that Corollary \ref{bavula} does not hold for $R$.
From Theorem \ref{Ggth} and (ii) of Theorem \ref{f2th} with $f_1 \in \{1,x_1\}$ and
$f_2 = x_1^2 + 1$, it follows that $\R(x_1) = \R(x_1/1)$ is not generated by an element of 
$\R(x_1)$ which is integral over $\R[1/(x_1^2+1),x_1/(x_1^2+1)]$. 
Hence Corollary \ref{bavula} does not hold for $R$. 

Furthermore, $R$ is not contained in $\R[p/q]$ for any $p,q \in \R[x_1]$. Indeed, suppose
that $R \subseteq \R[p/q]$ for some $p,q \in \R[x_1]$. By Theorem 
\ref{pquniq}, it follows from $\R(p/q) = \R(x_1/1)$ that $\deg q \le 1$. This contradicts
the fact that $x_1^2 + 1$ does not decompose into linear factors over $\R$. 
\end{example}

We end this section with a problem for which we do not know the answer.

\begin{problem}
Let $p, q \in K[x]$ which are not both constant, such that $\gcd\{p,q\} = 1$.
Suppose that $K(p/q)$ is algebraically closed in $K(x)$.
Is $\lambda p + q$ reducible for only finitely many $\lambda \in K$?
\end{problem}

\begin{proposition}
The above problem has an affirmative answer if $K$ is algebraically closed. 
\end{proposition}

\begin{proof}
Suppose that $K$ is algebraically closed and that $\lambda p + q$ is reducible over $K$ 
for infinitely many $\lambda \in K$. On account of Corollary 3 of 
\cite[\S~3.1]{MR1770638}, $\lambda p + q$ is reducible over $K$ for all 
$\lambda \in K$ such that $\deg (\lambda p + q) = \deg (p,q)$. 

Notice that $tp + q$ is irreducible over $K[t]$, and hence over $K(t)$ as well,
because $\gcd\{p,q\} = 1$. From the Bertini-Krull theorem, see Theorem 37 of 
\cite[\S~3.3]{MR1770638}, it follows that 
$tp + q = th^{*}_1(p^{*}\!,q^{*}) + h^{*}_2(p^{*}\!,q^{*})$ 
for some $p^{*}\!, q^{*} \in K[x]$ and a homogeneous
$h^{*} \in K[y_1,y_2]^2$, such that $\deg (p,q) > \deg (p^{*}\!,q^{*})$. Hence
$(p,q) = h^{*}(p^{*}\!,q^{*})$. From (1) $\Rightarrow$ (2) 
of Lemma \ref{hf}, we deduce that
$$
p/q = \frac{h^{*}_1(p^{*}\!,q^{*})}{h^{*}_2(p^{*}\!,q^{*})}
= \frac{h^{*}_1(p^{*}\!/q^{*},1)}{h^{*}_2(p^{*}\!/q^{*},1)} \in K(p^{*}\!/q^{*})\mbox{,}
$$
so $p^{*}\!/q^{*}$ is algebraic over $K(p/q)$. By assumption, $p^{*}\!/q^{*} \in K(p/q)$,
so $K(p/q) = K(p^{*}\!/q^{*})$. It follows from Proposition \ref{pquniq} that
either $\gcd\{p,q\} \ne 1$ or $\deg (p,q) \le \deg (p^{*}\!,q^{*})$. A contradiction.
\end{proof}

\section{A generalization of a theorem of Paul Gordan and Max N{\"o}ther}

We assume from now on that $m = n$. Therefore $y = (y_1,y_2,\ldots,y_n)$.

The following result is useful if we see a rational map as a product of 
a rational function and a polynomial map.

\begin{proposition} \label{gquasi}
Assume $g \in K(x) \setminus \{0\}$ and $H \in K(x)^n$. Then 
$$
\jac H \cdot H = \tr \jac H \cdot H 
~\Longleftrightarrow~ \jac (gH) \cdot gH = \tr \jac (gH) \cdot gH\mbox{.}
$$
\end{proposition}

\begin{proof}
Using the chain rule for differentation, we can deduce that
$$
\jac (gH) = H \cdot \jac g + g \jac H
\qquad \mbox{and} \qquad
\tr \jac (gH) = \jac g \cdot H + g \tr \jac H\mbox{,}
$$
from which it follows that
$$
\jac (gH) \cdot H - g \jac H \cdot H 
= H \cdot \jac g \cdot H
= H \cdot \tr \jac (gH) - H \cdot g \tr \jac H\mbox{.}
$$
Consequently,
$$
\jac (gH) \cdot H - H \cdot \tr \jac (gH)
= g \jac H \cdot H - H \cdot g \tr \jac H\mbox{.}
$$
We can rewrite this as
$$
g^{-1} \cdot \big(\jac (gH) \cdot (gH) - \tr \jac (gH) \cdot (gH)\big) = 
g \cdot (\jac H \cdot H - \tr \jac H \cdot H)\mbox{.}
$$
This gives the desired result.
\end{proof}

In Theorem \ref{gnnotrace} below, we classify all rational maps $H$ which satisfy 
$\jac H \cdot H = \tr \jac H \cdot H$, under the assumption that $H$ is homogeneous and 
$\trdeg_{K} (H) \le 2$. But Gordan and N{\"o}ther classified all polynomial maps
$H$ which satisfy $\jac H \cdot H = 0$, under the assumption that $H$ is 
homogeneous, $\trdeg_{K} (H) \le 2$ and the characteristic of $K$ is zero. The following
proposition shows that we indeed generalize the result of Gordan and N{\"o}ther.

\begin{proposition} \label{gngen}
Let $H \in K(x)^n$. Then the following hold.
\begin{enumerate}

\item[\upshape (i)] If $K$ has characteristic zero, then $\jac H \cdot H = 0$ implies $H(x + tH) = H$.

\item[\upshape (ii)] If $H \in K[x]^n$, then $H(x + tH) = H$ implies that $\jac H$ is nilpotent.
In particular $\jac H \cdot H = \tr H \cdot H = 0$.

\end{enumerate}
\end{proposition}

\begin{proof}
(i) follows from (3) $\Rightarrow$ (2) of \cite[Prop.~1.3]{1501.05168}.

The first claim of (ii) follows from (2) $\Rightarrow$ (1.7) of \cite[Prop.~1.3]{1501.05168}, 
because the condition that $K$ has characteristic zero is not used in the proof of that.
The last claim follows because $\jac H \cdot H$
is the coefficient of $t^1$ of $H(x + tH) - H$, and $\tr \jac H = 0$ if $\jac H$ is nilpotent.
\end{proof}

In Example \ref{gncounter} below, we show that the conditions that $K$ has characteristic zero
and that $H \in K[x]$ are necessary in (i) and (ii) of Proposition \ref{gngen} respectively.
Hence it was inevitable to replace the condition that $\jac H \cdot H = 0$ of Gordan and 
N{\"o}ther in the following theorem. This is done in (3) of Theorem \ref{gnnotrace}
below, in which $\jac (h(p,q)) \cdot h(p,q) = 0$ is the only assertion which does
not follow from the condition $\trdeg_K(tH) \le 2$.

\begin{theorem} \label{gnnotrace}
Assume that $H \in K(x)^n$ such that $\trdeg_K(tH) \le 2$. Let $\tilde{H}$ be the primitive
part of $H$. Then the following statements are equivalent:
\begin{enumerate}

\item[\upshape (1)] $\jac H \cdot H = \tr \jac H \cdot H$;

\item[\upshape (2)] $\jac \tilde{H} \cdot \tilde{H}(y) = \tr \jac \tilde{H} \cdot \tilde{H}(y) = 0$;

\item[\upshape (3)] There exist $h \in K[y_1,y_2]$ which is homogeneous or zero, 
such that 
$$
\deg h > 0 \Longrightarrow c \nmid \deg h\mbox{,}
$$
where $c$ is the characteristic of $K$, and $p,q \in K[x]$, such that 
$\gcd\{p,q\} = 1$, $H = g \cdot h(p,q)$ for some $g \in K(x)$, 
and 
$$
\jac (h(p,q)) \cdot h(p,q) = 0\mbox{;}
$$

\item[\upshape (4)] There exist $f \in K[y_1]$, and $p,q \in K[x]$, 
such that $H = g \cdot f(p/q)$ for some $g \in K(x)$
and 
$$
\jac p \cdot f = \jac q \cdot f = 0\mbox{;}
$$

\item[\upshape (5)] There exist $f$ and $p,q$ as in {\upshape(4)}, 
such that in addition, $\gcd\{f_1,f_2,\ldots,\allowbreak f_n\} = 1$ 
and $\gcd\{p,q\} = 1$.

\end{enumerate}
Furthermore, $H$ is a $K[x]$-linear combination of $n - \rk \jac \tilde{H}$ vectors
over $K$ if any of {\upshape (1)}, {\upshape (2)},  {\upshape (3)},  {\upshape (4)}
and {\upshape (5)} is satisfied.
\end{theorem}

\begin{proof}
The case $H = 0$ is easy, so assume from now on that $H \ne 0$.
We prove the equivalence of (1), (2), (3), (4) and (5) by way of six implications.
\begin{description}

\item[(2) \imp (1)] 
Assume (2). By substituting $y = x$ in (2), we obtain $\jac \tilde{H} \cdot \tilde{H} = 
\tr \jac \tilde{H} \cdot \tilde{H}$. Now (1) follows from Lemma \ref{gquasi}.

\item[(4) \imp (1)]
Assume (4), take $s \ge \deg f$ and define $h(y_1,y_2) := y_2^s f(y_1/y_2)$. 
From (2) $\Rightarrow$ (1) of Lemma \ref{hf}, it follows that $h \in K[y_1,y_2]$
is homogeneous of degree $s$. By substituting $y_1 = p(y)/q(y)$ in (4), we obtain 
$\jac p \cdot f\big(p(y)/q(y)\big) = \jac q \cdot f\big(p(y)/q(y)\big) = 0$, so
$$
\jac p \cdot h\big(p(y),q(y)\big) = \jac q \cdot h\big(p(y),q(y)\big) = 0\mbox{.}
$$
Define $\hat{H} := h(p,q)$. Since the row space of $\jac \hat{H}$ is generated by 
$\jac p$ and $\jac q$, it follows that
\begin{equation} \label{hatHy}
\jac \hat{H} \cdot \hat{H}(y) = 0\mbox{.}
\end{equation}
If we take the Jacobian with respect to $y$ on both sides, then we see that
$(\jac \hat{H})^2 = 0$, so $\tr \jac \hat{H} = 0$. Hence (2) with $\hat{H}$ instead
of $\tilde{H}$ follows. Now we can apply the proof of (2) $\Rightarrow$ (1)
above, but with $\hat{H}$ instead of $\tilde{H}$, to obtain (1).

\item[(5) \imp (2)]
Assume (5), take $s = \deg f$ and define $h(y_1,y_2) := y_2^s f(y_1/y_2)$. 
From (2) $\Rightarrow$ (1) of Lemma \ref{hf}, it follows that $h \in K[y_1,y_2]$
is homogeneous of degree $s$. Just as in the proof of (4) $\Rightarrow$ (1)
above, we can deduce that (2) is satisfied if we replace $\tilde{H}$ by 
$\hat{H} = h(p,q)$.

So it suffices to show that $h(p,q)$ is primitive. From Lemma \ref{hf}, 
it follows that $\gcd\{h_1,h_2,\ldots,h_n\} = 1$. 
Since $(p,q)$ is superprimitive, we deduce from Theorem \ref{ph} that 
$\gcd\{h_1(p,q),h_2(p,q),\ldots,h_n(p,q)\} = 1$. Hence $h(p,q)$ is indeed
primitive.

\item[(5) \imp (3)]
Assume (5) and take $s \ge \deg f$ such that
$$
s > 0 \Longrightarrow c \nmid s\mbox{,}
$$
where $c$ is the characteristic of $K$. Define $h(y_1,y_2) := y_2^s f(y_1/y_2)$.
From (2) $\Rightarrow$ (1) of Lemma \ref{hf}, it follows that $h \in K[y_1,y_2]$ is 
homogeneous of degree $s$. From \eqref{hatHy}, we deduce that
$\jac \hat{H} \cdot \hat{H} = 0$, where $\hat{H} = h(p,q)$. This yields (3).

\item[(1) \imp (5)]
Assume (1). Take $h$ and $p,q$ as in Theorem \ref{hmgrk2}.
Then $h$ is homogeneous, say of degree $s$, and $\gcd\{h_1,h_2,\ldots, h_n\} = 1$.
Furthermore, $\gcd\{p,q\} = 1$ and we may assume without loss of generality that 
$\deg p \le \deg q$. Let $f := h(y_1,1)$.
From Lemma \ref{hf}, it follows that $f(p/q) = q^{-s} h(p,q)$ and 
$\gcd\{f_1,f_2,\ldots,f_n\} = 1$. 

Define $f' := \parder{}{y_1} f$. Suppose first that $f(p/q)$ and $f'(p/q)$ are 
independent over $K(x)$. From Proposition \ref{gquasi}, it follows that 
$\jac \big(f(p/q)\big) \cdot f(p/q) = \tr \jac\big(f(p/q)\big) \cdot f(p/q)$, 
which is equivalent to
\begin{align*}
f'(p/q) \cdot \jac (p/q) \cdot f(p/q) 
&= f(p/q) \cdot \tr \big(f'(p/q) \cdot \jac (p/q)\big) \\
&= f(p/q) \cdot \jac (p/q) \cdot f'(p/q)\mbox{.}
\end{align*}
Since $f'(p/q)$ and $f(p/q)$ are independent over $K(x)$, their coefficients
$\jac (p/q) \cdot f(p/q)$ and $\jac (p/q) \cdot f'(p/q)$ are zero. Now (5) follows from
(i) of Lemma \ref{flem} below.

Suppose next that $f(p/q)$ and $f'(p/q)$ are dependent over $K(x)$. We show 
that we can choose $p$ and $q$ such that $\jac p = \jac q = 0$, which yields (4).
This is clear if $p/q \in K$, so assume that $p/q \notin K$. Then $f$ and $f'$ 
are dependent over $K(y_1)$. Hence $f_i f' = f'_i f$ for each $i \in \{1,2,\ldots,n\}$. 
Since $K[x]$ is a PSP-domain, it follows from $\gcd\{f_1,f_2,\ldots,f_n\} = 1$ that 
$\gcd\{f'_if_1,f'_if_2,\ldots,f'_if_n\} = f'_i$. Consequently, $f_i \mid f'_i$ for each $i$.

On the other hand, $\deg f_i > \deg f'_i$ if $f_i \ne 0$, so we can deduce that
$f' = 0$. Consequently, $f \in K[y_1^c]^n$, where $c$ is the characteristic of $K$. 
By replacing $p$ and $q$ by $p^c$ and $q^c$ and adapting
$f$ accordingly if $c > 0$, we obtain $\jac p = \jac q = 0$ indeed. 

\item[(3) \imp (4)]
Assume (3). Let $s$ be the degree of $h$ and define $f := h(y_1,1)$ and 
$f' := \parder{}{y_1} f$. Assume without loss of generality that $\deg p \le \deg q$.

Suppose first that $f(p/q)$ and $f'(p/q)$ are independent over $K(x)$. Then
\begin{align*}
0 &= \jac \big(h(p,q)\big) \cdot h(p,q)\\
&= q^s \jac \big(f(p/q)\big) \cdot q^s f(p/q) + f(p/q) \cdot \jac (q^s) \cdot q^s f(p/q) \\
&= q^{2s} f'(p/q) \cdot \jac (p/q) \cdot f(p/q) + s q^{2s-1} f(p/q) \cdot \jac q \cdot f(p/q)\mbox{.}
\end{align*}
Since $f'(p/q)$ and $f(p/q)$ are independent over $K(x)$, the factors 
$\jac (p/q) \cdot f(p/q)$ and $s \jac q \cdot f(p/q)$ of their coefficients are zero. 
Furthermore, $s \ge \deg f > \deg f' \ge 0$, so $c \nmid s$ by assumption.
Hence $s \ne 0$ in $K$ and $\jac (p/q) \cdot f(p/q) = 0 = \jac q \cdot f(p/q)$. 
Now (4) follows from (ii) of Lemma \ref{flem} below.

Suppose next that $f(p/q)$ and $f'(p/q)$ are dependent over $K(x)$. If $p/q \in K$ or 
$\gcd\{f_1,f_2,\ldots,f_n\} = 1$, then we can proceed as in the proof of 
(1) $\Rightarrow$ (5) to obtain (4). So assume that $p/q \notin K$ and that 
$\gcd\{f_1,f_2,\ldots,f_n\} \ne 1$. Then $f$ and $f'$ are dependent over $K(y_1)$.

Since $K[y_1]$ is a GCD-domain, we can replace $f$ by its primitive part and adapt $g$
accordingly, because the product rule of differentiation tells us that $f$ and $f'$ will 
remain dependent over $K(y_1)$. 
So $\gcd\{f_1,f_2,\ldots,f_n\} = 1$, and we can proceed as in the proof of 
(1) $\Rightarrow$ (5) above to obtain (4).

\end{description}
To prove the last claim of this theorem, assume that any of (1), (2), (3), (4) and (5) 
is satisfied. Then (2) is satisfied, so that
$$
\jac \tilde{H} \cdot \tilde{H}(y) = 0 = \jac \tilde{H} \cdot H(y)\mbox{.}
$$
Let $H^{(\alpha)}$ be the vector over $K$ of coefficients of 
$y_1^{\alpha_1}y_2^{\alpha_2}\cdots y_n^{\alpha_n}$ of $H(y)$. 
From $\jac \tilde{H} \cdot H(y) = 0$, it follows that 
$\jac \tilde{H} \cdot H^{(\alpha)} = 0$. As $\dim \ker \jac \tilde{H} = 
n - \rk \jac \tilde{H}$, there can only be $n - \rk \jac \tilde{H}$
independent vectors $H^{(\alpha)}$.
\end{proof}

\begin{lemma} \label{flem}
Let $f \in K[y_1]^n$ and suppose that $p,q \in K[x]$, such that $\gcd\{p,q\} = 1$ and
$\deg q \ge \deg p$. Then $\jac p \cdot f = \jac q \cdot f = 0$ in the following cases.
\begin{enumerate}

\item[\upshape (i)] $\jac (p/q) \cdot f(p/q) = \jac (p/q) \cdot f'(p/q) = 0$, 
where $f' = \parder{}{y_1} f$,

\item[\upshape (ii)] $\jac (p/q) \cdot f(p/q) = \jac q \cdot f(p/q) = 0$.

\end{enumerate}
\end{lemma}

\begin{proof}
If $q \in K$, then $p \in K$ as well, and $\jac p \cdot f = \jac q \cdot f = 0$ is
trivially satisfied. So assume that $q \notin K$. Let $s \ge \deg f$ and write 
\begin{alignat*}{7}
f &=\;& v_{s+1} y_1^{s+1} &\,+\,& v_s y_1^s &\,+\,& \cdots 
&\,+\,& v_1 y_1 &\,+\,& v_0 &\,+\,& v_{-1} y_1^{-1}&
\qquad \mbox{and} \\
y_1 f' &=\;& (s+1) v_{s+1} y_1^{s+1} &\,+\,& s v_s y_1^s &\,+\,& \cdots 
&\,+\,& v_1 y_1 && &\,-\,& v_{-1} y_1^{-1}\mbox{,}&
\end{alignat*}
where $v_i \in K^n$ for all $i$. Notice that $v_{s+1} = v_{-1} = 0$. 
\begin{enumerate}

\item[(i)] Assume that $\jac (p/q) \cdot f(p/q) = \jac (p/q) \cdot f'(p/q) = 0$. We first show 
that $\jac p \cdot v_i = \jac q \cdot v_{i-1}$ for all $i$. Hence suppose that there exists 
an $i$ such that $\jac p \cdot v_i \ne \jac q \cdot v_{i-1}$. Take such an $i$ as large as possible. 
From $\jac (p/q) \cdot f(p/q) = 0$, it follows that
\begin{align*}
0 &= q^2 \jac (p/q) \cdot q^{s+1} f(p/q) \\
&= (q \jac p - p \jac q) \cdot 
\big(v_{s+1} p^{s+1} + v_s p^sq + \cdots + v_1 pq^s + v_0 q^{s+1}\big) \\
&\equiv q \jac p \cdot v_i p^i q^{s+1-i} - 
        p \jac q \cdot v_{i-1}p^{i-1} q^{s+2-i} \pmod{q^{s+3-i}}\mbox{.}
\end{align*}
Consequently, $q^{s+3-i} \mid p^i q^{s+2-i} (\jac p \cdot v_i - \jac q \cdot v_{i-1})$.

Since $\gcd\{p,q\} = 1$, we deduce that 
$q \mid (\jac p \cdot v_i - \jac q \cdot v_{i-1})$. By comparing degrees on both sides, 
we see that $\jac p \cdot v_i - \jac q \cdot v_{i-1} = 0$, which contradicts the definition 
of $i$. So $\jac p \cdot v_i = \jac q \cdot v_{i-1}$ for all $i$. 

Using $\jac (p/q) \cdot (p/q)f'(p/q) = 0$ instead of $\jac (p/q) \cdot f(p/q) = 0$, 
we can deduce that $\jac p \cdot i v_i = \jac q \cdot (i-1) v_{i-1}$ for all $i$ as well. 
Combining both, we conclude that $\jac p \cdot v_i = \jac q \cdot v_i = 0$ for all $i$, 
which yields $\jac p \cdot f = \jac q \cdot f = 0$.

\item[(ii)] Assume that $\jac (p/q) \cdot f(p/q) = \jac q \cdot f(p/q) = 0$. 
It suffices to show that $\jac p \cdot v_i = \jac q \cdot v_i = 0$ for all $i$.
Hence suppose that there exists an $i$ such that either $\jac p \cdot v_i \ne 0$
or $\jac q \cdot v_i \ne 0$. Take such an $i$ as large as possible. 
From $\jac q \cdot f(p/q) = 0$, it follows that
\begin{align*}
0 &= \jac q \cdot q^{s+1} f(p/q) \\
&= \jac q \cdot \big(v_{s+1} p^{s+1} + v_s p^sq + \cdots + v_1 pq^s + v_0 q^{s+1}\big) \\
&\equiv \jac q \cdot v_i p^i q^{s+1-i} \pmod{q^{s+2-i)}}\mbox{.}
\end{align*}
Consequently, $q^{s+2-i} \mid p^i q^{s+1-i} \jac q \cdot v_i$.
Since $\gcd\{p,q\} = 1$, we deduce that $q \mid \jac q \cdot v_i$.
By comparing degrees on both sides, we see that $\jac q \cdot v_i = 0$.

Since $\jac p = q\jac (p/q) + (p/q) \jac q$, it follows 
from $\jac (p/q) \cdot f(p/q) = \jac q \cdot f(p/q) = 0$ that 
$\jac p \cdot f(p/q) = 0$ as well. Hence $\jac p \cdot v_i = 0$ can be
deduced in a similar manner as $\jac q \cdot v_i = 0$ if $p \notin K$. 
If $p \in K$, then $\jac p \cdot v_i = 0$ in any case. This contradicts
the definition of $i$.
\qedhere

\end{enumerate}
\end{proof}

The reader may verify that (3) of Theorem \ref{gnnotrace} can be replaced by 
$\jac \tilde{H} \cdot \tilde{H} = 0$ if the characteristic of $K$ is zero.
The following example shows that this is not the case for positive characteristic.

\begin{example} \label{gncounter}
The map $H = (1,x_2/x_1)$ satisfies $H(x+tH) = H$ and hence also $\jac H \cdot H = 0$,
but $\jac H \cdot H(y) \ne 0$ and $\jac H$ is not nilpotent. In particular, the condition that
$H \in K[x]^n$ is necessary in (ii) of Proposition \ref{gngen}.

The map $\tilde{H} = (x_1^s H,x_2^s) = 
(x_1^s, x_1^{s-1}x_2, x_2^s)$ does not satisfy $\tilde{H}(x+t\tilde{H}) = \tilde{H}$, 
but does satisfy $\jac \tilde{H} \cdot \tilde{H} = 0$, provided $s$ is divisible by the 
characteristic of $K$. In particular, the condition that $K$ has characteristic zero
is necessary in (i) of Proposition \ref{gngen}.

But just as with $H$, $\jac \tilde{H} \cdot \tilde{H}(y) \ne 0$ and $\jac \tilde{H}$ 
is not nilpotent. Hence the condition that the degree of $h$ is not divisible by the 
characteristic of $K$ in (3) of Theorem \ref{gnnotrace} is necessary.
\end{example}

%\cite{DBLP:conf/mfcs/2016}

\bibliographystyle{hmgrk2qt}
\bibliography{hmgrk2qt}

\begin{thebibliography}{dBvdE}

\bibitem[AZ]{MR2310957}
D.~D. Anderson and M.~Zafrullah.
\newblock The {S}chreier property and {G}auss' lemma.
\newblock {\em Boll. Unione Mat. Ital. Sez. B Artic. Ric. Mat. (8)},
  10(1):43--62, 2007.

\bibitem[AM]{MR0321934}
M.~Artin and D.~Mumford.
\newblock Some elementary examples of unirational varieties which are not
  rational.
\newblock {\em Proc. London Math. Soc. (3)}, 25:75--95, 1972.

\bibitem[Bav]{MR2134288}
V.~Bavula.
\newblock L\"uroth field extensions.
\newblock {\em J. Pure Appl. Algebra}, 199(1-3):1--10, 2005.

\bibitem[dB1]{MR3177043}
Michiel de~Bondt.
\newblock The strong nilpotency index of a matrix.
\newblock {\em Linear Multilinear Algebra}, 62(4):486--497, 2014.

\bibitem[dB2]{1501.05168}
Michiel de~Bondt.
\newblock Quasi-translations and singular hessians.
\newblock arXiv:1501.05168, 2015.

\bibitem[dBvdE]{MR2179727}
Michiel de~Bondt and Arno van~den Essen.
\newblock The {J}acobian conjecture: linear triangularization for homogeneous
  polynomial maps in dimension three.
\newblock {\em J. Algebra}, 294(1):294--306, 2005.

\bibitem[Cas]{MR1510838}
Guido Castelnuovo.
\newblock Sulla razionalit\`a delle involuzioni piane.
\newblock {\em Math. Ann.}, 44(1):125--155, 1894.

\bibitem[CG]{MR0302652}
C.~Herbert Clemens and Phillip~A. Griffiths.
\newblock The intermediate {J}acobian of the cubic threefold.
\newblock {\em Ann. of Math. (2)}, 95:281--356, 1972.

\bibitem[Deb]{Debarre2014}
Olivier Debarre.
\newblock Around cubic hypersurfaces.
\newblock \url{http://www.math.ens.fr/~debarre/Duesseldorf2014.pdf}, 2014.

\bibitem[Eak]{MR0289498}
Paul Eakin.
\newblock A note on finite dimensional subrings of polynomial rings.
\newblock {\em Proc. Amer. Math. Soc.}, 31:75--80, 1972.

\bibitem[FMN]{DBLP:conf/mfcs/2016}
Piotr Faliszewski, Anca Muscholl, and Rolf Niedermeier, editors.
\newblock {\em 41st International Symposium on Mathematical Foundations of
  Computer Science, {MFCS} 2016, August 22-26, 2016 - Krak{\'{o}}w, Poland},
  volume~58 of {\em LIPIcs}. Schloss Dagstuhl - Leibniz-Zentrum fuer
  Informatik, 2016.

\bibitem[For]{MR0913158}
Edward Formanek.
\newblock Two notes on the {J}acobian conjecture.
\newblock {\em Arch. Math. (Basel)}, 49(4):286--291, 1987.

\bibitem[GN]{MR1509898}
P.~Gordan and M.~N{\"{o}}ther.
\newblock Ueber die algebraischen {F}ormen, deren {H}esse'sche {D}eterminante
  identisch verschwindet.
\newblock {\em Math. Ann.}, 10(4):547--568, 1876.

\bibitem[Gor]{MR1510417}
Paul Gordan.
\newblock Ueber biquadratische {G}leichungen.
\newblock {\em Math. Ann.}, 29(3):318--326, 1887.

\bibitem[Igu]{MR0048860}
Jun-ichi Igusa.
\newblock On a theorem of {L}ueroth.
\newblock {\em Mem. Coll. Sci. Univ. Kyoto Ser. A. Math.}, 26:251--253, 1951.

\bibitem[IM]{MR0291172}
V.~A. Iskovskih and Yu.~I. Manin.
\newblock Three-dimensional quartics and counterexamples to the {L}\"uroth
  problem.
\newblock {\em Mat. Sb. (N.S.)}, 86(128):140--166, 1971.

\bibitem[Lü]{MR1509855}
J.~Lüroth.
\newblock Beweis eines {S}atzes \"uber rationale {C}urven.
\newblock {\em Math. Ann.}, 9(2):163--165, 1875.

\bibitem[Mur]{MR0352089}
J.~P. Murre.
\newblock Reduction of the proof of the non-rationality of a non-singular cubic
  threefold to a result of {M}umford.
\newblock {\em Compositio Math.}, 27:63--82, 1973.

\bibitem[Oja]{Ojanguren1990}
Manuel Ojanguren.
\newblock The {W}itt group and the problem of {L}\"uroth.
\newblock \url{http://sma.epfl.ch/~ojangure/WL.pdf}, 1990.

\bibitem[PSS]{DBLP:conf/mfcs/PandeySS16}
Anurag Pandey, Nitin Saxena, and Amit Sinhababu.
\newblock Algebraic independence over positive characteristic: New criterion
  and applications to locally low algebraic rank circuits.
\newblock In Faliszewski et~al. \cite{DBLP:conf/mfcs/2016}, pages 74:1--74:15.

\bibitem[Sam]{MR0058251}
Pierre Samuel.
\newblock Some remarks on {L}\"uroth's theorem.
\newblock {\em Mem. Coll. Sci. Univ. Kyoto. Ser. A. Math.}, 27:223--224, 1953.

\bibitem[Sch]{MR1770638}
A.~Schinzel.
\newblock {\em Polynomials with special regard to reducibility}, volume~77 of
  {\em Encyclopedia of Mathematics and its Applications}.
\newblock Cambridge University Press, Cambridge, 2000.
\newblock With an appendix by Umberto Zannier.

\bibitem[Seg]{MR0042746}
Beniamino Segre.
\newblock Sull'esistenza, sia nel campo razionale che nel campo reale, di
  involuzioni piane non birazionali.
\newblock {\em Atti Accad. Naz. Lincei. Rend. Cl. Sci. Fis. Mat. Nat. (8)},
  10:94--97, 1951.

\bibitem[Ste]{MR1580791}
Ernst Steinitz.
\newblock Algebraische {T}heorie der {K}\"orper.
\newblock {\em J. Reine Angew. Math.}, 137:167--309, 1910.

\bibitem[Zak]{MR0280471}
Abraham Zaks.
\newblock Dedekind subrings of {$k[x_{1},\cdots ,x_{n}]$} are rings of
  polynomials.
\newblock {\em Israel J. Math.}, 9:285--289, 1971.

\bibitem[Zar]{MR0099990}
Oscar Zariski.
\newblock On {C}astelnuovo's criterion of rationality {$p_{a}=P_{2}=0$} of an
  algebraic surface.
\newblock {\em Illinois J. Math.}, 2:303--315, 1958.

\end{thebibliography}

\end{document}